\documentclass[final,3p,fleqn,times]{elsarticle}

\usepackage{amssymb}
\usepackage{amsthm}
\usepackage{mathtools}
\newtheorem{remark}{Remark}

\usepackage{lipsum}
\usepackage{amsfonts}
\usepackage{graphicx}
\usepackage{epstopdf}
\usepackage{lineno}
\usepackage{amsmath}
\ifpdf
  \DeclareGraphicsExtensions{.eps,.pdf,.png,.jpg}
\else
  \DeclareGraphicsExtensions{.eps}
\fi
\usepackage{algorithm}
\usepackage{algpseudocode}

\usepackage{stmaryrd} 

\usepackage{tikz}
\usepackage{pgfplots}
\usetikzlibrary{patterns}
\pgfplotsset{compat=newest}
\usetikzlibrary{calc}

\pgfplotsset{
  colormap={rainbowDesaturated}{
      rgb255=( 71, 71,219)
      rgb255=(  0,  0, 92)
      rgb255=(  0,255,255)
      rgb255=(  0,128,  0)
      rgb255=(255,255,  0)
      rgb255=(255, 97,  0)
      rgb255=(107,  0,  0)
      rgb255=(224, 77, 77)
    },
}

\usepackage{subcaption}
\usepackage{tabularx}
\usepackage{multirow}
\usepackage{multicol}
\usepackage{booktabs}

\usepackage[colorlinks=true, allcolors=blue]{hyperref}
\usepackage[capitalise,sort&compress]{cleveref}
\biboptions{sort&compress}

\renewcommand{\vec}[1]{\mathbf{#1}}
\newcommand{\tensorOne}[1]{\boldsymbol{#1}}
\newcommand{\tensorTwo}[1]{\boldsymbol{#1}}
\newcommand{\tensorFour}[1]{\mathbb{#1}}
\newcommand{\divergence}{\operatorname{div}}

\newcommand{\evaluate}[3]{\left.{#1}\right|_{#2}^{#3}}

\renewcommand{\Vec}[1]{\boldsymbol{\mathbf{#1}}}
\newcommand{\Mat}[1]{#1}
\newcommand{\diam}[1]{\operatorname{diam(#1)}}
\newcommand{\meas}[1]{\operatorname{meas(#1)}}
\newcommand{\normH}[1]{\|{#1}\|_{H^1}}
\newcommand{\snormH}[1]{|{#1}|_{1}}
\newcommand{\normL}[1]{\|{#1}\|_{L^2}}
\newcommand{\normLg}[1]{\|{#1}\|_{L^2(\Gamma)}}

\newtheorem{theorem}{Theorem}

\newtheorem{lemma}{Lemma}

\definecolor{viridis0}{rgb}{0.267004, 0.004874, 0.329415}
\definecolor{WildStrawberry}{rgb}{0.933333, 0.160784, 0.403922}

\newcommand{\logLogSlopeTriangle}[5]
{

  \pgfplotsextra{
    \pgfkeysgetvalue{/pgfplots/xmin}{\xmin}
    \pgfkeysgetvalue{/pgfplots/xmax}{\xmax}
    \pgfkeysgetvalue{/pgfplots/ymin}{\ymin}
    \pgfkeysgetvalue{/pgfplots/ymax}{\ymax}

    \pgfmathsetmacro{\xArel}{#1}
    \pgfmathsetmacro{\yArel}{#3}
    \pgfmathsetmacro{\xBrel}{#1-#2}
    \pgfmathsetmacro{\yBrel}{\yArel}
    \pgfmathsetmacro{\xCrel}{\xArel}

    \pgfmathsetmacro{\lnxB}{\xmin* (1- (#1-#2))+\xmax* (#1-#2)} 
    \pgfmathsetmacro{\lnxA}{\xmin* (1-#1)+\xmax*#1} 
    \pgfmathsetmacro{\lnyA}{\ymin* (1-#3)+\ymax*#3} 
    \pgfmathsetmacro{\lnyC}{\lnyA+#4* (\lnxA-\lnxB)}
    \pgfmathsetmacro{\yCrel}{(\lnyC-\ymin)/ (\ymax-\ymin)} 

    \coordinate(A) at (rel axis cs:\xArel,\yArel);
    \coordinate(B) at (rel axis cs:\xBrel,\yBrel);
    \coordinate(C) at (rel axis cs:\xCrel,\yCrel);

    \draw[#5]   (A)-- node[pos=0.5,anchor=north] {1}
    (B)--
    (C)-- node[pos=0.5,anchor=west] {#4}
    cycle;
  }
}

\usepackage{tikz-3dplot}

\journal{arXiv}

\begin{document}

\begin{frontmatter}

  \title{A Robust Framework for Frictional Fault Contact in Geological Formations Using a Stabilized Augmented Lagrangian Approach}

  \author[inst1]{ Matteo Frigo }

  \affiliation[inst1]{organization={Stanford University},
    addressline={367 Panama Street},
    city={Stanford},
    postcode={94305},
    state={CA},
    country={USA}}

  \author[inst2]{Nicola Castelletto}
  \author[inst2]{Matteo Cusini}
  \author[inst2]{Randolph Settgast}
  \author[inst1]{Hamdi Tchelepi}

  \affiliation[inst2]{organization={Lawrence Livermore National Laboratory},
    addressline={7000 East Ave},
    city={Livermore},
    postcode={94550},
    state={CA},
    country={USA}}

  \begin{abstract}
    Numerical simulations are essential for evaluating the performance and safety of geological engineered systems such as geologic carbon storage sites, enhanced geothermal fields, and oil and gas reservoirs. A key challenge lies in accurately modeling the frictional contact behavior along fault surfaces. This problem involves inequality constraints that arise from the physics of frictional slip, requiring specialized numerical methods to handle the resulting highly nonlinear and path-dependent behavior.
    In this work, we address this challenge using an Augmented Lagrangian Method (ALM) implemented via the Uzawa algorithm. The formulation employs mixed finite element spaces, combining low-order piecewise linear displacements within the 3D domain cells with piecewise constant tractions defined on the fault surfaces. Furthermore, to ensure stability and satisfy the inf-sup condition, the discrete displacement space is enriched with face bubble functions on both sides of the contact interfaces. This approach offers several advantages over other stabilization techniques that rely on additional terms, as it does not require extra parameters for implementation, and it integrates naturally in the Uzawa framework.

  \end{abstract}

\end{frontmatter}

\section{Introduction}
\label{sec:intro}
Geologic carbon dioxide, geothermal energy production, and other subsurface energy applications face significant geomechanical challenges \cite{FerGamJanTea10,JuaHagHer12,TeaCasGam14,BuiEtal19,KivPujRutVil22,SalSilLunRogDavJua25}.
A key concern is ensuring the safety of anthropogenic operations while avoiding unwanted outcomes such as induced seismicity, fault reactivation, and damage the rock integrity in the surrounding formation.
In this context, numerical modeling of contact mechanics serves as a valuable tool for devising risk-mitigation strategies  and supporting the decision-making process.

Even though the theoretical framework of contact mechanics is well-defined, simulating these processes remains challenging due to their highly nonlinear nature. Additionally, numerical convergence is furthered hindered  by  complex subsurface stratigraphy and heterogeneous physical properties in the domain.
The requirement to accurately represent geological layers and the high levels of heterogeneity can lead to distorted grids, characterized by poor aspect ratios and significant property variations between elements. All of these factors contribute to the creation of a highly challenging problem, both from the perspectives of nonlinear and linear solving.
For these reasons, the research and development of robust and reliable numerical methods capable of addressing these issues are of paramount importance.

Many approaches have been developed over the years to address contact mechanics \cite{WriLau06}, and providing an exhaustive review is beyond the scope of this work. However, we briefly summarize the main strategies, beginning with a classification based on a key distinguishing feature: how each approach represents discontinuities.
The first class of discretization methods relies on a conforming grid that explicitly resolves the geometry of the contact surfaces, enabling the use of standard finite element techniques  \cite{PusLau04,FraFerJanTea16,GarKarTch16,SetFuWalWhi17}. A potential disadvantage of this approach is that, in presence of  a complex fault geometry or fracture network, the grid can easily become excessively refined and distorted.
The second class consists of embedded methods, which treat the contact surface as independent surfaces overlaid on a separate background grid. The continuum discretization of the elements intersected by this surface is enriched to capture discontinuities \cite{OliHueSan06,Bor08,RenJiaYou16,CusWhiCasSet21}. However, extracting accurate stress fields near these discontinuities can be challenging due to the use of enrichment functions and non-standard interpolation.
A third class of methods eliminates the need for a sharp representation of the discontinuity, instead modeling fractures as regularized (smooth) fields within a continuum framework. This  class includes phase-field and damage-mechanics-based approaches \cite{WheWicWol14,GeeLiuHuMicDol19}.

In this work, we adopt a conforming discretization. This choice is motivated by the need for accurate traction values along contact surfaces, which are critical for assessing fault reactivation and informing models of induced seismicity.
Among the available strategies within this class, the Penalty Method (PM) and the Lagrange Multiplier Method (LMM) are two of the most commonly used techniques to enforce contact constraints. Each comes with its own advantages and limitations \cite{Yas13,WriLau06}. The Penalty Method is straightforward and often easier to implement; however, it may yield inaccurate solutions in cases of strong nonlinearities or result in ill-conditioned linear systems. Conversely, the Lagrange Multiplier Method offers a more accurate enforcement of constraints but introduces additional global degrees of freedom and typically requires an active set strategy \cite{HueWoh05} to handle the inequality conditions, making it more computationally demanding. In this context, the Augmented Lagrangian Method (ALM) emerges as a compelling alternative, effectively bridging the gap between the two by combining the advantages of both approaches while alleviating many of their respective limitations  \cite{Roc74,CurAla88,AlaCur91,SimLau92,BurHanLar19}.
In this paper, we propose an Augmented Lagrangian Method (ALM) approach that utilizes a mixed formulation consisting of lower-order linear piecewise displacement alongside a constant face-centered Lagrange multiplier. This choice of discrete spaces raises concerns regarding inf-sup stability that must be addressed to ensure accurate simulations.
One class of stabilization techniques addresses inf-sup stability in mixed formulations by enriching the displacement space, typically through the use of second-order elements or the incorporation of bubble functions \cite{BrePit84,BreLeoMarRus97,BreMar01,HauLeT07,PusLauSol08,BofBreFor13,RodHuOhmAdlGasZik18,DroEncFaiHaiMas24}. Although effective, this approach has the potential drawback of introducing additional degrees of freedom and resulting in denser linear systems.
A second class of methods achieves stabilization by modifying the constraint equations with an additional term. This additional contribution may be derived from macroelement theory or through physics-based arguments \cite{ElmSilWat14,FraCasWhiTch20,FriCasFerWhi21,AroCasHamWhiTch24}. In both cases, the added term is scaled by parameters that can be difficult to determine , particularly when the mesh used to discretize the domain is unstructured and exhibits severe heterogeneity.
In this work, we propose to enhance the displacement space using face bubble functions, which ensure stability without relying on additional parameters. Furthermore, the incorporation of bubble functions not only facilitates a more natural integration process within the Uzawa procedure of the Augmented Lagrangian method, but, in the specific context of contact mechanics, they also offer an additional advantage: when static condensation is applied, they do not increase the fill-in of the global system matrix.

The manuscript is organized as follows. The next section provides a brief summary of the mathematical formulation. This is followed by a description of the face-bubble function stabilization strategy and a proof of the inf-sup condition. We then present several numerical test cases to demonstrate the effectiveness and robustness of the proposed method. Finally, the conclusion section summarizes the key points of the paper.

\section{Mathematical formulation}
\label{sec:matform}
\subsection{Problem statement}
Let us consider the time domain of interest $\mathbb{T} = (0, t_{\max})$ and a bounded regular domain $\overline{\Omega} = \Omega \cup \partial \Omega$ in $\mathbb{R}^3$, with $\Omega$ an open set, $\partial \Omega$ its sufficiently smooth boundary, and $\tensorOne{n}_{\Omega}$ its outer normal vector.
The boundary is divided into two distinct, non-overlapping subsets, such that $\partial \Omega = {\Gamma_u \cup \Gamma_{\sigma}}$ and $\Gamma_u \cap \Gamma_{\sigma} = \emptyset$.
Dirichlet boundary conditions for displacement are applied on $\Gamma_u$, while Neumann boundary conditions for traction are applied on $\Gamma_{\sigma}$.
The contact surface $\Gamma_f$ is mathematically represented as two sufficiently regular overlapping surfaces, $\Gamma_f^{-}$ and $\Gamma_f^{+}$, embedded within the domain $\Omega$.
The orientation of the internal boundary is characterized by a unit vector orthogonal to the fracture plane; by convention, we define $\tensorOne{n}_f=\tensorOne{n}_f^{-}=-\tensorOne{n}_f^{+}$.
Finally, we introduce the trace operators $\gamma^{-}$ and $\gamma^{+}$, which enable the restriction of functions defined on $\Omega$ to the internal boundaries $\Gamma_f^{-}$ and $\Gamma_f^{+}$, respectively.
Given these preliminaries, under the assumption of quasi-static conditions and infinitesimal strains, the strong form of the initial boundary value problem (IBVP) describing the deformation of the domain $\overline{\Omega}$ can be stated as \cite{WriLau06, Kik88}:
find the displacement, $\tensorOne{u}$ : $\overline{\Omega} \times \mathbb{T} \rightarrow \mathbb{R}^3$, and the traction fields, $\tensorOne{t}$ : $\Gamma_f \times \mathbb{T} \rightarrow \mathbb{R}^3$, such that:
\begin{linenomath}
  \begin{subequations}
    \begin{align}
      - \divergence{ \tensorTwo{\sigma}(\tensorOne{u})}
       & = \tensorOne{0}
       &                 &
      \mbox{ in } \Omega \times \mathbb{T}
       &                 &
      \mbox{(linear momentum balance)},
      \label{eq:momentumBalanceS}
      \\
      \llbracket \tensorTwo{\sigma}(\tensorOne{u}) \rrbracket \cdot \tensorOne{n}
       & =
      \tensorOne{0}
       &                 &
      \mbox{ in } \Gamma_f \times \mathbb{T}
       &                 &
      \mbox{(fracture traction balance)},
      \label{eq:tractionBalanceS}
    \end{align}\label{eq:IBVP}\null
  \end{subequations}
\end{linenomath}
subject to the constraints:
\begin{linenomath}
  \begin{subequations}
    \begin{align}
      t_N
      =
      \tensorOne{t} \cdot \tensorOne{n}_f                                                     & \le 0
                                                                                              &       &
      \mbox{ in } \Gamma_f \times \mathbb{T}
                                                                                              &       &
      \mbox{(normal contact conditions)},
      \label{eq:normalcontact0}
      \\
      g_N
      =
      \llbracket \tensorOne{u} \rrbracket \cdot \tensorOne{n}_f                               & \ge 0
                                                                                              &       &
      \mbox{ in } \Gamma_f \times \mathbb{T},
                                                                                              &       &
      \label{eq:normalcontact1}
      \\
      t_N g_N                                                                                 & = 0
                                                                                              &       &
      \mbox{ in } \Gamma_f \times \mathbb{T},
                                                                                              &       &
      \label{eq:normalcontact2}
      \\
      \| \tensorOne{t}_T \|_2 - \tau_{\max}                                                   & \le 0
                                                                                              &       &
      \mbox{ in } \Gamma_f \times \mathbb{T}
                                                                                              &       &
      \mbox{(Coulomb frictional law)},
      \label{eq:coulomb0}
      \\
      \dot{\tensorOne{g}}_T \cdot \tensorOne{t}_T - \tau_{\max} \| \dot{\tensorOne{g}}_T \|_2 & = 0
                                                                                              &       &
      \mbox{ in } \Gamma_f \times \mathbb{T},
                                                                                              &       &
      \label{eq:coulomb1}
    \end{align}\label{eq:constraints}\null
  \end{subequations}
\end{linenomath}
and a set of boundary and initial conditions:
\begin{linenomath}
  \begin{subequations}
    \begin{align}
      \tensorOne{u}
       & =
      \overline{\tensorOne{u}}
       &   &
      \mbox{ on } \Gamma_u \times \mathbb{T},
       &   &
      \\
      \tensorTwo{\sigma}x(\tensorOne{u})\cdot \tensorOne{n}_{\Omega}
       & =
      \overline{\tensorOne{t}}
       &   &
      \mbox{ on } \Gamma_{\sigma} \times \mathbb{T},
       &   &
      \\
      \tensorOne{u}
       & =
      {\tensorOne{u}_0}
       &   &
      \mbox{ on } \Omega.
       &   &
    \end{align}\label{eq:bc}\null
  \end{subequations}
\end{linenomath}
Here, $\tensorTwo{\sigma}$ is the Cauchy stress tensor and it is related to the displacement vector through the constitutive relationship $\tensorTwo{\sigma} = \tensorFour{C} : \nabla^s \tensorOne{u}$, where $\tensorFour{C}$ is the fourth-order elasticity tensor and $\nabla^s$ is the symmetric gradient operator.
$\tensorOne{t}$ is the traction vector acting on $\Gamma_f$, defined as
$\tensorOne{t} = \tensorTwo{\sigma} \cdot \tensorOne{n}^{-}_f = - \tensorTwo{\sigma} \cdot \tensorOne{n}^{+}_f = (t_N \tensorOne{n}_f + \tensorOne{t}_T )$, where $t_N$ and $\tensorOne{t}_T$ represent the normal and tangential components, with respect to the local reference system.
The limit value $\tau_{\max} = c - t_N \tan(\theta)$ is the maximum admissible modulus of $\tensorOne{t}_T$ according to the Coulomb criterion, with $c$ and $\theta$ the cohesion and friction angle, respectively.
The linear continuous operator $\llbracket \cdot \rrbracket$ denotes the jump of a quantity across the fracture surface $\Gamma_f$, defined as $\llbracket \cdot \rrbracket = ( \gamma^{+}(\cdot) - \gamma^{-} (\cdot) )$, with $\gamma^{\pm}(\cdot)$ the trace operator onto $\Gamma_f^{+}$ and $\Gamma_f^{-}$, respectively.
For instance, the jump in the displacement field is expressed as $\llbracket \tensorOne{u} \rrbracket = ( \gamma^{+}(\tensorOne{u}) - \gamma^{-} (\tensorOne{u}) ) = ( g_N \tensorOne{n}_f + \tensorOne{g}_T)$, representing the relative displacement across $\Gamma_f$, where $g_N$ and $\tensorOne{g}_T$ are the normal and tangential components of the displacement jump, respectively.

\subsection{Weak formulation}
Assuming, without loss of generality, homogeneous Dirichlet boundary conditions, a classical weak formulation can be derived by introducing the following functional spaces:
\begin{linenomath}
  \begin{subequations}
    \begin{align}
      \boldsymbol{\mathcal{U}}
       & =
      \{
      \tensorOne{v} \in \boldsymbol{H}^1(\Omega)
      :
      \tensorOne{v}|_{\Gamma_u}=\tensorOne{0}
      \},
      \label{eq:spaceU}
      \\
      \boldsymbol{\mathcal{M}} \left( \tensorOne{t} \right )
       & =
      \{
      \tensorOne{\mu} \in \boldsymbol{H}^{-\frac{1}{2}}(\Gamma_f)
      :
      \mu_N \le 0, \;
      \left( \tensorOne{\mu}, \tensorOne{w} \right)_{\Gamma_f} \le  \left( \tau_{\max} \left(t_N \right), \| \tensorOne{w}_T \|_2 \right )_{\Gamma_f},
      \forall \tensorOne{w} \in \boldsymbol{H}^{\frac{1}{2}}(\Gamma_f), \;
      w_N \ge 0
      \}. \label{eq:spaceM}
    \end{align}\label{eq:functionSpaces}\null
  \end{subequations}
\end{linenomath}
Then the weak formulation of the IBVP~\eqref{eq:IBVP} reads: for each $t \in \mathbb{T} $ find
$(\tensorOne{u}, \tensorOne{t}) \in \boldsymbol{\mathcal{U}} \times \boldsymbol{\mathcal{M}}$
such that
\begin{linenomath}
  \begin{subequations}
    \begin{alignat}{2}
       & (\nabla^s \tensorOne{v},\tensorTwo{\sigma}(\tensorOne{u}))_{\Omega}
      + ( \llbracket \tensorOne{v} \rrbracket, \tensorOne{t}_N + \tensorOne{t}_T )_{\Gamma_f}
      - ( \tensorOne{v}, \bar{\tensorOne{t}} )_{\Gamma_{\sigma}}
      = 0,
       &                                                                                   &
      \quad
      \forall \tensorOne{v} \in \boldsymbol{\mathcal{U}},
      \label{eq:momentumBalanceW}
      \\
       & ( \tensorOne{t}_N - \tensorOne{\mu}, \tensorOne{g}_N (\tensorOne{u}) )_{\Gamma_f}
      + ( \tensorOne{t}_T - \tensorOne{\mu}, \dot{\tensorOne{g}}_T (\tensorOne{u}) )_{\Gamma_f}
      \ge 0,
      &                                                                                   &
      \quad
      \forall \tensorOne{\mu} \in \boldsymbol{\mathcal{M}}.
      \label{eq:tractionBalanceW}
    \end{alignat}\label{eq:weakform}\null
  \end{subequations}
\end{linenomath}
The compact notation $(*, \star)_D$ is used to denote the $L^2$-inner product on a given domain $D$ of scalar, vector, or rank-2 tensor functions in $L^2(D)$, $[L^2(D)]^3$, of $[L^2(D)]^{3\times 3}$, as appropriate -- for example, $( \llbracket \tensorOne{v} \rrbracket, {\tensorOne{t}} )_{\Gamma_f} = \int_{\Gamma_f}  \llbracket \tensorOne{v} \rrbracket \cdot \tensorOne{t}_N \; \mathrm{d}\Gamma $.
Note the apparent dimensional inconsistency in \cref{eq:tractionBalanceW}, where the first argument of the first and second inner product has units of displacement and velocity, respectively.
Since static Coulomb friction is considered, the rate of the tangential displacement $\dot{\tensorOne{g}}_T$ should be interpreted in a discrete setting as a pseudo-time dependent quantity.
By discretizing the time interval $\mathbb{T}$ into $n_T$  sub-intervals of size $\Delta t = (t_n - t_{n-1})$, $\dot{\tensorOne{g}}_T$ can be replaced in the discrete setting by the tangential displacement increment $\Delta_n \tensorOne{g}_T$ ~\cite{Woh11}, where $\Delta_n$ denotes the backward difference operator defined as $\Delta_n (\bullet) = (\bullet)_n - (\bullet)_{n-1}$, with $n$ and $(n-1)$ indicating the current and the previously converged discrete time levels, respectively.

The weak formulation \eqref{eq:weakform} constitutes a non-linear problem that involves inequalities.
Solving inequalities into a weak formulation can present significant challenges.
For this reason, many modern numerical strategies for contact mechanics aim to reformulate the original problem in terms of variational equalities, which can be easily integrated into a finite element framework and do not necessitate entirely new minimization techniques.
Various approaches can be employed to achieve this objective.
Among these, the active set strategy---a well-established technique in quadratic programming---is particularly noteworthy, often used with Lagrange multipliers.
This method enables the selection of an active contact region, allowing the replacement of inequality \eqref{eq:tractionBalanceW} with a variational equation.
Once the non-linear problem has been solved, the status of the fracture element is assessed against the assumed region, and a new active set is selected accordingly.
This process continues iteratively until convergence of the contact set is achieved.
However, this procedure has some disadvantages. In fact, there is no guarantee regarding the speed of convergence, nor is there assurance of convergence itself.
Furthermore, it has been observed that as problems become more complex, the method may start to loop excessively, leading to multiple resolutions of the non-linear problem, resulting in increased computational costs.
Another widely used strategy for addressing inequalities in the weak formulation is the Augmented Lagrangian Method \cite{SimLau92, Ren13}.
This approach yields a smooth energy functional and transforms the problem into a fully unconstrained one, ensuring the exact satisfaction of contact constraints with a finite penalty parameter.

Let us revisit the fundamentals of the Augmented Lagrangian Method.
The key idea involves replacing the inner products in \cref{eq:tractionBalanceW} using the following relationships:
\begin{linenomath}
  \begin{subequations}
    \begin{alignat}{2}
       & (\tensorOne{\mu}, \tensorOne{t}_N)_{\Gamma_f}
      = (\tensorOne{\mu}, ( {t}_N + \varepsilon_N {g}_N )_{-} \tensorOne{n}_f)_{\Gamma_f},
      &                                                                                   &
      \quad
      \forall \tensorOne{\mu} \in \boldsymbol{\mathcal{M}},     
      \label{eq:lagmultNWeak}
      \\
       & (\tensorOne{\mu}, \tensorOne{t}_T)_{\Gamma_f}
      = ( \tensorOne{\mu}, P_{(0, \tau_{\max}(t_N, g_N))} (\tensorOne{t}_T + \varepsilon_{T} \dot{\tensorOne{g}}_T ) )_{\Gamma_f},
      &                                                                                   &
      \quad
      \forall \tensorOne{\mu} \in \boldsymbol{\mathcal{M}}.
      \label{eq:lagmultTWeak}
    \end{alignat}\label{eq:lagmultWeak}\null
  \end{subequations}
\end{linenomath}
where $\varepsilon_N$ and $\varepsilon_T$ are the positive augmentation (i.e. stiffness) parameters for normal and tangential contact, respectively, and $\tau_{\max}(t_N, g_N) = ( c - \tan(\theta) ( t_N + \varepsilon_N g_N )_{-} ) $.
The term $P_{(0,\rho)}(\tensorOne{t}) = \min\left( 1, \frac{\rho}{||\tensorOne{t}||} \right) \tensorOne{t} $ represents the orthogonal projection onto the closed ball centered at 0 with radius $\rho$, while the notation $(\cdot)_{-}$ denotes the function defined as $(x)_{-} = x$ for $x\leq 0$ and $(x)_{-}=0$ for $x > 0$.
\cref{eq:lagmultNWeak} is equivalent to the relations governing normal contact (\cref{eq:normalcontact0,eq:normalcontact1,eq:normalcontact2}):
\begin{itemize}
  \item
        if $( {t}_N + \varepsilon_N {g}_N ) > 0$, then it follows from \cref{eq:lagmultNWeak} that $\tensorOne{t}_N=0$.
        Hence, $g_N > 0$ must be satisfied according to the initial assumption.
        This situation corresponds to an open contact condition.
  \item
        If $(t_N + \varepsilon_N g_N) \leq 0$, then from \cref{eq:lagmultNWeak}, we have $\tensorOne{t}_N = (t_N + \varepsilon_N g_N) \tensorOne{n}_f$.
        This implies that $g_N = 0$ and corresponds to a closed contact condition, which can occur in either stick or slip mode.
\end{itemize}
Similarly, \cref{eq:lagmultTWeak} is equivalent to the tangential friction conditions given in  \cref{eq:coulomb0,eq:coulomb1}:
\begin{itemize}
  \item
        if $\|\tensorOne{t}_T + \varepsilon_T \dot{\tensorOne{g}}_T \|_2 \leq \tau_{\max} (t_N, g_N)$, then it follows from \cref{eq:lagmultTWeak} that $\tensorOne{t}_T = \tensorOne{t}_T + \varepsilon_T \dot{\tensorOne{g}}_T$, which implies $\dot{\tensorOne{g}}_T=0$, i.e. the stick contact condition.
  \item
        If $||\tensorOne{t}_T + \varepsilon_T \dot{\tensorOne{g}}_T ||_2 > \tau_{\max} (t_N,g_N)$, the following implicit vector equation for $\tensorOne{t}_T$ is obtained
        \begin{align}
           & \tensorOne{t}_T = \tau_{\max}(t_N,g_N) \frac{\tensorOne{t}_T + \varepsilon_T \dot{\tensorOne{g}}_T}{\|\tensorOne{t}_T + \varepsilon_T \dot{\tensorOne{g}}_T\|_2},
          \label{eq:collinearity1}
        \end{align}
        which implies $||\tensorOne{t}_T||_2 = \left\|\tau_{\max}(t_N,g_N) \frac{\tensorOne{t}_T + \varepsilon_T \dot{\tensorOne{g}}_T}{\|\tensorOne{t}_T + \varepsilon_T \dot{\tensorOne{g}}_T\|_2}\right\|_2 =  \tau_{\max}(t_N,g_N) $. Expressing the tangential traction as $\tensorOne{t}_T = \tau_{\max}(t_N,g_N) \tensorOne{m}_T$, with $\tensorOne{m}_T = \frac{\tensorOne{t}_T}{||\tensorOne{t}_T||_2} $, \cref{eq:collinearity1} can be rearranged as
        \begin{align}
          (|| \tau_{\max}(t_N,g_N) \tensorOne{m}_T + \varepsilon_T \dot{\tensorOne{g}}_T  ||_2 - \tau_{\max}(t_N,g_N)) \tensorOne{m}_T = \varepsilon_T \dot{\tensorOne{g}}_T
          \label{eq:collinearity2},
        \end{align}
        which shows the collinearity between $\tensorOne{t}_T$ and $\dot{\tensorOne{g}}_T $. If $\tensorOne{t}_T$ and $\dot{\tensorOne{g}}_T$ are collinear vectors in the opposite direction, i.e. $\tensorOne{t}_T = - \tau_{\max}(t_N,g_N) \frac{\dot{\tensorOne{g}}_T}{||\dot{\tensorOne{g}}_T||_2}$,  \cref{eq:collinearity2} is satisfied only if $\varepsilon_T ||\dot{\tensorOne{g}}_T||_2 < \tau_{\max}(t_N,g_N) $. Instead, assuming $\tensorOne{t}_T$ and $\dot{\tensorOne{g}}_T$ collinear vectors in the same direction as in \cref{eq:coulomb1}, \cref{eq:collinearity2} is always valid for all values of $\tau_{\max}(t_N,g_N)$ and $\varepsilon_T$. To ensure the uniqueness of the solution, we always choose the parameter $\varepsilon_T$ such that the condition $\varepsilon_T > \frac{\tau_{\max}(t_N,g_N)}{||\dot{\tensorOne{g}}_T||_2}$ is satisfied.
\end{itemize}
\cref{eq:lagmultNWeak,eq:lagmultTWeak} allow the reformulation of \eqref{eq:weakform} as an unconstrained saddle-point problem.
Assuming a converged solution at discrete time level $(n-1)$, the continuous weak form \eqref{eq:weakform} at time level $n$ can be restated as:
find
$(\tensorOne{u}_n, \tensorOne{t}_n) \in \boldsymbol{\mathcal{U}} \times \boldsymbol{\mathcal{M}}$
such that
\begin{linenomath}
  \begin{subequations}
    \begin{alignat}{2}
       & (\nabla^s \tensorOne{v},\tensorTwo{\sigma}(\tensorOne{u}_n))_{\Omega}
      + ( \llbracket \tensorOne{v} \rrbracket, \hat{\tensorOne{t}}_{N}(\tensorOne{u}_n, \tensorOne{t}_{n} ) + \hat{\tensorOne{t}}_{T} ( \tensorOne{u}_n, \tensorOne{u}_{n-1}, \tensorOne{t}_{n} ) )_{\Gamma_f}
      - ( \tensorOne{v}, \bar{\tensorOne{t}} )_{\Gamma_{\sigma}}
      = 0,
       &                                                                                                                                             &
      \quad
      \forall \tensorOne{v} \in \boldsymbol{\mathcal{U}},
      \label{eq:momentumBalanceWALM}                                                                                                                   \\
       & -\frac{1}{\varepsilon_N} ( \tensorOne{\mu}, \tensorOne{t}_{N,n} - \hat{\tensorOne{t}}_{N}(\tensorOne{u}_n, \tensorOne{t}_{n} ) )_{\Gamma_f}
      - \frac{1}{\varepsilon_T}(\tensorOne{\mu}_T, \tensorOne{t}_{T,n} - \hat{\tensorOne{t}}_T( \tensorOne{u}_n, \tensorOne{u}_{n-1}, \tensorOne{t}_{n} ) )_{\Gamma_f}
      = 0,
       &                                                                                                                                             &
      \quad \forall \tensorOne{\mu} \in \boldsymbol{\mathcal{M}},
      \label{eq:constraintsWALM}
    \end{alignat}\label{eq:weakformALM}\null
  \end{subequations}
\end{linenomath}
where $\hat{\tensorOne{t}}_{N}$ and $\hat{\tensorOne{t}}_T$ are the augmented Lagrange multipliers defined as follows:
\begin{linenomath}
  \begin{subequations}
    \begin{align}
       & \hat{\tensorOne{t}}_{N}
      = \hat{\tensorOne{t}}_{N}(\tensorOne{u}_n, \tensorOne{t}_{n} )
      = (\tensorOne{n_f} \cdot ( \tensorOne{t}_{N,n}+\varepsilon_N \tensorOne{g}_N(\tensorOne{u}_n) )_{-} \tensorOne{n_f}, \label{eq:lagmultN} \\
       & \hat{\tensorOne{t}}_{T}
      = \hat{\tensorOne{t}}_T( \tensorOne{u}_n, \tensorOne{u}_{n-1}, \tensorOne{t}_{n} )
      = P_{(0, \tau_{\max}(\hat{t}_N) )}
      (\tensorOne{t}_{T,n} + \varepsilon_{T} \Delta_n{\tensorOne{g}}_T(\tensorOne{u}_n, \tensorOne{u}_{n-1}) ). \label{eq:lagmultT}
    \end{align}\label{eq:lagmult}\null
  \end{subequations}
\end{linenomath}
The weak form \eqref{eq:weakformALM} yields the same solution as the original problem \eqref{eq:weakform}, with the advantage of being fully unconstrained.
Hence, traditional numerical methods can be employed for the solution of \cref{{eq:momentumBalanceWALM,eq:constraintsWALM}} without requiring an active set strategy.
Moreover, \cref{eq:lagmultT} suggests the update procedure for the Lagrange multiplier, which can be employed to solve the problem \eqref{eq:weakformALM} through an iterative sequential scheme, as we will discuss in detail in Section \ref{sec:disc}.
The existence and uniqueness of the solution to the problem \eqref{eq:weakformALM} are established using standard techniques \cite{Kik88, Jar83}.
Some existence results for problems involving Coulomb friction have been demonstrated in \cite{EckJar98} under reasonable assumptions regarding the regularity of the boundary and for sufficiently small friction coefficients \cite{LabRen08}.
However, the uniqueness of the solution for any friction coefficients remains an open problem.

\section{Discretization and solution strategy}
\label{sec:disc}
\subsection{Discretization}
We numerically solve the system of equations \eqref{eq:weakformALM} using a mixed finite-element approach for the spatial discretization.
Let us introduce a triangulation $\mathcal{T}$ of the domain $\overline{\Omega}$, consisting of non-overlapping hexahedral, tetrahedral, and wedge elements that conform to the contact surface $\Gamma_f$, such that $\overline{\Omega} = \bigcup_{K \in \mathcal{T}} K $.
Also, let us define $\mathcal{F}_f$ as the set of quadrilateral and triangular faces $\varphi$ in the triangulation such that $\Gamma_f = \bigcup_{\varphi \in \mathcal{F}_f} {\varphi}$.
Given the triangulation $\mathcal{T}$, we introduce the following finite-dimensional subspaces approximating $\boldsymbol{\mathcal{U}}$ and $ \boldsymbol{\mathcal{M}}$:
\begin{linenomath}
  \begin{subequations}
    \begin{align}
      \boldsymbol{\mathcal{U}}^{h,1}
       & = \{
      \tensorOne{v}^{h} \in \boldsymbol{\mathcal{U}}
      :
      \tensorOne{v}^{h}|_K \in [\mathbb{X}_1(K)]^3, \; \forall K \in \mathcal{T}
      \},
      \label{eq:functionSpace_Uh1}
      \\
      \boldsymbol{\mathcal{M}}^{h}
       & =\{
      \tensorOne{\mu}^{h} \in \boldsymbol{\mathcal{M}}
      :
      \tensorOne{\mu}^{h}|_{\varphi} \in {[\mathbb{P}_0(\varphi)]^{3}}, \; \forall \varphi \in \mathcal{F}_{f}
      \}.
      \label{eq:spaceM_h}
    \end{align}
  \end{subequations}
\end{linenomath}
In \eqref{eq:functionSpace_Uh1}, $\mathbb{X}_1(K)$ represents the space of trilinear nodal basis functions $\mathbb{Q}_1(K)$ for hexahedral elements, the space of linear nodal basis functions $\mathbb{P}_1(K)$ for tetrahedral elements, and for wedge elements the tensor product space of linear nodal basis functions on the triangular face and linear nodal basis functions along the wedge height.
In \eqref{eq:spaceM_h}, three piecewise-constant vector basis functions are associated with each face used to mesh $\Gamma_f$, namely $\{ \chi_{\varphi} \tensorOne{n}_f, \chi_{\varphi} \tensorOne{m}_1, \chi_{\varphi} \tensorOne{m}_2 \}$, with $\chi_{\varphi}$ the characteristic function of the face $\varphi \in \mathcal{F}_f$, and $\tensorOne{m}_1$ and $\tensorOne{m}_2$ to unit vector that, together with $\tensorOne{n}_f$, form a complete orthonormal local reference frame.
We will denote by $\{ \tensorOne{\eta}_i \}$ and $\{ \tensorOne{\mu}_i \}$ the sets of vector basis functions for $\boldsymbol{\mathcal{U}}^{h,1}$ and $\boldsymbol{\mathcal{M}}^{h}$, respectively.
Figure \ref{fig:discretization} illustrates the degrees of freedom (dof) for an hexahedral mesh with interface elements representing a discontinuity (i.e. a fault).

The fully discrete weak form of \eqref{eq:weakformALM} is expressed as:
find $(\tensorOne{u}^h_n, \tensorOne{t}^h_n ) \in \boldsymbol{\mathcal{U}}^{h,1} \times \boldsymbol{\mathcal{M}}^{h}$ such that
\begin{linenomath}
  \begin{subequations}
    \begin{alignat}{2}
       & (\nabla^s \tensorOne{v}^h,\tensorTwo{\sigma}(\tensorOne{u}^h_n))_{\Omega}
      + ( \llbracket \tensorOne{v}^h \rrbracket, \hat{\tensorOne{t}}_{N}(\tensorOne{u}^h_n, \tensorOne{t}^h_{n} ) + \hat{\tensorOne{t}}_{T} ( \tensorOne{u}^h_n, \tensorOne{u}^h_{n-1}, \tensorOne{t}^h_{n} ) )_{\Gamma_f}
      - ( \tensorOne{v}^h, \bar{\tensorOne{t}} )_{\Gamma_{\sigma}}
      = 0,
       &                                                                                                                                                     &
      \quad
      \forall \tensorOne{v}^h \in \boldsymbol{\mathcal{U}}^{h,1},
      \label{eq:momentumBalanceWALMdisc}                                                                                                                       \\
       & -\frac{1}{\varepsilon_N} ( \tensorOne{\mu}^h, \tensorOne{t}^h_{N,n} - \hat{\tensorOne{t}}_{N}(\tensorOne{u}^h_n, \tensorOne{t}^h_{n} ) )_{\Gamma_f}
      - \frac{1}{\varepsilon_T}(\tensorOne{\mu}_T, \tensorOne{t}^h_{T,n} - \hat{\tensorOne{t}}_T( \tensorOne{u}^h_n, \tensorOne{u}^h_{n-1}, \tensorOne{t}^h_{n} ) )_{\Gamma_f}
      = 0,
       &                                                                                                                                                     &
      \quad \forall \tensorOne{\mu}^h \in \boldsymbol{\mathcal{M}}^h.
      \label{eq:weakformALMdisc}
    \end{alignat}
    \label{eq:weakformDisc}\null
  \end{subequations}
\end{linenomath}

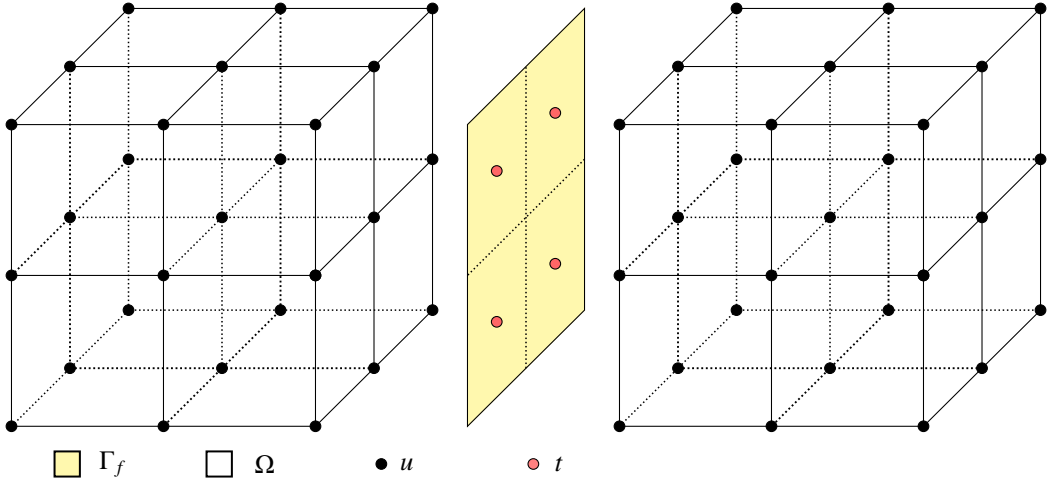
\begin{figure}[h]
  \centering
  \begin{tikzpicture}[scale=1]
    \def\slice{2.0}
    \def\side{4}

    \filldraw[color=yellow!40] (2.0 + \side,0,0) -- (2.0 + \side,\side,0) -- (2.0 + \side,\side,\side) -- (2.0 + \side,0,\side) -- cycle;

    \draw[line width=0.25mm,densely dotted] (0,\slice,\side) -- (0,\slice,0) -- (\side,\slice,0);
    \draw (0,\slice,\side) -- (\side,\slice,\side) -- (\side,\slice,0);

    \draw[line width=0.25mm,densely dotted] (4.0 + \side + 0,\slice,\side) -- (4.0 + \side + 0,\slice,0) -- (4.0 + \side + \side,\slice,0);
    \draw (4.0 + \side + 0,\slice,\side) -- (4.0 + \side + \side,\slice,\side) -- (4.0 + \side + \side,\slice,0);

    \draw[line width=0.25mm,densely dotted] (\side,0,\slice) -- (0,0,\slice) -- (0,\side,\slice);
    \draw (\side,0,\slice) -- (\side,\side,\slice) -- (0,\side,\slice);

    \draw[line width=0.25mm,densely dotted] (4.0 + \side + \side,0,\slice) -- (4.0 + \side + 0,0,\slice) -- (4.0 + \side + 0,\side,\slice);
    \draw (4.0 + \side + \side,0,\slice) -- (4.0 + \side + \side,\side,\slice) -- (4.0 + \side + 0,\side,\slice);

    \draw[line width=0.25mm,densely dotted] (\slice,\side,0) -- (\slice,0,0) -- (\slice,0,\side);
    \draw (\slice,\side,0) -- (\slice,\side,\side) -- (\slice,0,\side);

    \draw[line width=0.25mm,densely dotted] (4.0 + \side + \slice,\side,0) -- (4.0 + \side + \slice,0,0) -- (4.0 + \side + \slice,0,\side);
    \draw (4.0 + \side + \slice,\side,0) -- (4.0 + \side + \slice,\side,\side) -- (4.0 + \side + \slice,0,\side);

    \draw (\side,0,0) -- (\side,\side,0) -- (0,\side,0);
    \draw (0,0,\side) -- (\side,0,\side) -- (\side,\side,\side) -- (0,\side,\side) -- (0,0,\side);
    \draw (\side,0,0) -- (\side,0,\side);
    \draw (\side,\side,0) -- (\side,\side,\side);
    \draw (0,\side,0) -- (0,\side,\side);
    \draw[line width=0.2mm,densely dotted] (0,0,\side) -- (0,0,0) -- (\side,0,0);
    \draw[line width=0.2mm,densely dotted] (0,\side,0) -- (0,0,0);

    \draw (4.0 + \side + \side,0,0) -- (4.0 + \side + \side,\side,0) -- (4.0 + \side + 0,\side,0);
    \draw (4.0 + \side + 0,0,\side) -- (4.0 + \side + \side,0,\side) -- (4.0 + \side + \side,\side,\side) -- (4.0 + \side + 0,\side,\side) -- (4.0 + \side + 0,0,\side);
    \draw (4.0 + \side + \side,0,0) -- (4.0 + \side + \side,0,\side);
    \draw (4.0 + \side + \side,\side,0) -- (4.0 + \side + \side,\side,\side);
    \draw (4.0 + \side + 0,\side,0) -- (4.0 + \side + 0,\side,\side);
    \draw[line width=0.2mm,densely dotted] (4.0 + \side + 0,0,\side) -- (4.0 + \side + 0,0,0) -- (4.0 + \side + \side,0,0);
    \draw[line width=0.2mm,densely dotted] (4.0 + \side + 0,\side,0) -- (4.0 + \side + 0,0,0);

    \draw (2.0 + \side,0,0) -- (2.0 + \side,\side,0) -- (2.0 + \side,\side,\side) -- (2.0 + \side,0,\side) -- cycle;
    \draw[line width=0.2mm,densely dotted] (2.0 + \side + 0,\slice,0) -- (2.0 + \side + 0,\slice,\side);
    \draw[line width=0.2mm,densely dotted] (2.0 + \side + 0,0,\slice) -- (2.0 + \side + 0,\side,\slice);

    \draw[line width=0.2mm,densely dotted,color=black] (\slice,0,\slice) -- (\slice,\side,\slice);
    \draw[line width=0.2mm,densely dotted,color=black] (0,\slice,\slice) -- (\side,\slice,\slice);
    \draw[line width=0.2mm,densely dotted,color=black] (\slice,\slice,0) -- (\slice,\slice,\side);

    \draw[line width=0.2mm,densely dotted,color=black] (4.0 + \side + \slice,0,\slice) -- (4.0 + \side + \slice,\side,\slice);
    \draw[line width=0.2mm,densely dotted,color=black] (4.0 + \side + 0,\slice,\slice) -- (4.0 + \side + \side,\slice,\slice);
    \draw[line width=0.2mm,densely dotted,color=black] (4.0 + \side + \slice,\slice,0) -- (4.0 + \side + \slice,\slice,\side);

    \draw[fill=black] (\side,0,0) circle (0.2em);
    \draw[fill=black] (\side,\side,0) circle (0.2em);
    \draw[fill=black] (\side,\side,\side) circle (0.2em);
    \draw[fill=black] (\side,0,\side) circle (0.2em);
    \draw[fill=black] (0,0,\side) circle (0.2em);
    \draw[fill=black] (0,\side,\side) circle (0.2em);
    \draw[fill=black] (0,\side,0) circle (0.2em);

    \draw[fill=black] (\slice,\slice,\slice) circle (0.2em);
    \draw[fill=black] (0,\slice,\slice) circle (0.2em);
    \draw[fill=black] (\slice,0,\slice) circle (0.2em);
    \draw[fill=black] (\slice,\slice,0) circle (0.2em);
    \draw[fill=black] (\side,\slice,\slice) circle (0.2em);
    \draw[fill=black] (\slice,\side,\slice) circle (0.2em);
    \draw[fill=black] (\slice,\slice,\side) circle (0.2em);

    \draw[fill=black] (0,\side,\slice) circle (0.2em);
    \draw[fill=black] (\side,0,\slice) circle (0.2em);
    \draw[fill=black] (\side,\slice,0) circle (0.2em);
    \draw[fill=black] (\side,\slice,\side) circle (0.2em);
    \draw[fill=black] (\slice,\side,\side) circle (0.2em);
    \draw[fill=black] (\side,\slice,\side) circle (0.2em);
    \draw[fill=black] (\side,\side,\slice) circle (0.2em);
    \draw[fill=black] (\slice,0,\side) circle (0.2em);
    \draw[fill=black] (0,\slice,\side) circle (0.2em);
    \draw[fill=black] (0,\slice,0) circle (0.2em);
    \draw[fill=black] (0,0,\slice) circle (0.2em);
    \draw[fill=black] (\slice,0,0) circle (0.2em);
    \draw[fill=black] (\slice,\side,0) circle (0.2em);

    \draw[fill=black] (0,0,0) circle (0.2em);

    \draw[fill=black] (4.0 + \side + \side,0,0) circle (0.2em);
    \draw[fill=black] (4.0 + \side + \side,\side,0) circle (0.2em);
    \draw[fill=black] (4.0 + \side + \side,\side,\side) circle (0.2em);
    \draw[fill=black] (4.0 + \side + \side,0,\side) circle (0.2em);
    \draw[fill=black] (4.0 + \side + 0,0,\side) circle (0.2em);
    \draw[fill=black] (4.0 + \side + 0,\side,\side) circle (0.2em);
    \draw[fill=black] (4.0 + \side + 0,\side,0) circle (0.2em);

    \draw[fill=black] (4.0 + \side + \slice,\slice,\slice) circle (0.2em);
    \draw[fill=black] (4.0 + \side + 0,\slice,\slice) circle (0.2em);
    \draw[fill=black] (4.0 + \side + \slice,0,\slice) circle (0.2em);
    \draw[fill=black] (4.0 + \side + \slice,\slice,0) circle (0.2em);
    \draw[fill=black] (4.0 + \side + \side,\slice,\slice) circle (0.2em);
    \draw[fill=black] (4.0 + \side + \slice,\side,\slice) circle (0.2em);
    \draw[fill=black] (4.0 + \side + \slice,\slice,\side) circle (0.2em);

    \draw[fill=black] (4.0 + \side + 0,\side,\slice) circle (0.2em);
    \draw[fill=black] (4.0 + \side + \side,0,\slice) circle (0.2em);
    \draw[fill=black] (4.0 + \side + \side,\slice,0) circle (0.2em);
    \draw[fill=black] (4.0 + \side + \side,\slice,\side) circle (0.2em);
    \draw[fill=black] (4.0 + \side + \slice,\side,\side) circle (0.2em);
    \draw[fill=black] (4.0 + \side + \side,\slice,\side) circle (0.2em);
    \draw[fill=black] (4.0 + \side + \side,\side,\slice) circle (0.2em);
    \draw[fill=black] (4.0 + \side + \slice,0,\side) circle (0.2em);
    \draw[fill=black] (4.0 + \side + 0,\slice,\side) circle (0.2em);
    \draw[fill=black] (4.0 + \side + 0,\slice,0) circle (0.2em);
    \draw[fill=black] (4.0 + \side + 0,0,\slice) circle (0.2em);
    \draw[fill=black] (4.0 + \side + \slice,0,0) circle (0.2em);
    \draw[fill=black] (4.0 + \side + \slice,\side,0) circle (0.2em);

    \draw[fill=black] (4.0 + \side + 0,0,0) circle (0.2em);

    \draw[fill=red!60] (2.0 + \side,1.0,1.0) circle (0.2em);
    \draw[fill=red!60] (2.0 + \side,1.0,3.0) circle (0.2em);
    \draw[fill=red!60] (2.0 + \side,3.0,1.0) circle (0.2em);
    \draw[fill=red!60] (2.0 + \side,3.0,3.0) circle (0.2em);

    \node at (2*\side/6,-\slice/4,\side) {$\Gamma_{f}$};
    \node at (5*\side/6,-\slice/4,\side) {$\Omega$};
    \node at (7.8*\side/6,-\slice/4,\side) {$u$};
    \node at (10.8*\side/6,-\slice/4,\side) {$t$};
    \draw[black, line width=0.075em, fill=white] (3.85*\side/6,-1*\side/6,\side) rectangle (4.35*\side/6,-0.5*\side/6,\side);
    \draw[black, line width=0.075em, fill=yellow!40] (0.85*\side/6,-1*\side/6,\side) rectangle (1.35*\side/6,-0.5*\side/6,\side);
    \draw[fill=black] (7.3*\side/6,-\slice/4,\side) circle (0.2em);
    \draw[fill=red!50] (10.3*\side/6,-\slice/4,\side) circle (0.2em);
  \end{tikzpicture}
  \caption{Illustration of displacement $u$ (black dots) and traction degrees of freedom $t$  (red dots) on interface elements.}
  \label{fig:discretization}
\end{figure}

\subsection{Solution strategy}\label{subsec:solutionStrategy}
The discrete problem \eqref{eq:weakformDisc} forms a system of nonlinear equations.
A widely used strategy to solve this system is the Uzawa procedure~\cite{Lau03,Wri06}, an iterative optimization method that alternates between two steps.
Denoting the Uzawa iteration count by $k$, the first step involves computing the displacement $\tensorOne{u}_n^{h,k+1}$, given the discrete displacement from the previous timestep $\tensorOne{u}_{n-1}^{h}$ and the contact stress $\tensorOne{t}_n^{h,k}$, by solving the discrete linear momentum balance \cref{eq:momentumBalanceWALMdisc}.
The discrete problem can be expressed as a nonlinear system of algebraic equations which is solved for the current solution vector $\Vec{u} = \{ u_i \}$ --- the coefficient vector used to expand $\tensorOne{u}_n^{h,k+1}$ in terms of basis functions $\tensorOne{\eta}_i$ --- using Newton's method as follows.
Given an initial guess $\vec{u}_0$, for $\ell  = 0, 1, \ldots$, until convergence
\begin{linenomath}
  \begin{subequations}
    \begin{alignat}{2}
       & \text{solve}
       &              &
      \quad
      {\Mat{A}}_{uu}(\vec{u}_{\ell}) \delta \vec{u} = - \vec{r}_u(\vec{u}_{\ell}),
      \label{eq:NR1}
      \\
       & \text{set}
       &              &
      \quad
      \vec{u}_{\ell+1} = \vec{u}_{\ell} + \delta \vec{u}
      \label{eq:NR2}
    \end{alignat}
    \label{eq:NR1NR2}\null
  \end{subequations}
\end{linenomath}
where ${\Mat{A}}_{uu}(\vec{u}_{\ell}) = (\partial{\vec{r}_u}/\partial{\vec{u}})(\vec{u}_{\ell})$.
The matrix expressions $\vec{r}_u$ and $\Mat{A}_{uu}$ are provided in \ref{sec:appendixA}.

Using the solution to \eqref{eq:NR1NR2}, the second step of the Uzawa procedure involves correcting the Lagrange multipliers to enforce the contact constraints.
Based on~\cref{eq:weakformALMdisc}, and recalling that $\boldsymbol{\mathcal{M}}^h$ is a discrete space of piecewise constant functions, the updated tractions can be evaluated on each face $\varphi \in \mathcal{F}_f$  as a function of the traction at the previous Uzawa iteration $\tensorOne{t}_{n,\varphi}^{h,k} = \chi_{\varphi} ( t_{\varphi,N}^{k} \tensorOne{n}_f + t_{\varphi,1}^{k} \tensorOne{m}_1 + t_{\varphi,2}^{k} \tensorOne{m}_2 )$ and the average normal and tangential displacement jump
\begin{align}
  g_{N,\varphi}                      & = \frac{1}{|\varphi|} \int_{\varphi}  \tensorOne{g}_N(\tensorOne{u}_{n}^{h,k+1,\ell} ) \cdot \tensorOne{n}_f \; \mathrm{d}\varphi,
                                     &
  \Delta_n \tensorOne{g}_{T,\varphi} & = \frac{1}{|\varphi|} \int_{\varphi} \Delta_n \tensorOne{g}_T(\tensorOne{u}_{N}^{h,k+1,\ell}, \tensorOne{u}_{n-1}^{h} )  \; \mathrm{d}\varphi,
\end{align}
with $|\varphi|$ the area of face $\varphi$, using the following expressions:
\begin{linenomath}
  \begin{subequations}
    \begin{alignat}{2}
       & t_{\varphi,N}^{k+1}
      = (t_{\varphi,N}^{k} + \varepsilon_N g_{N,\varphi})_{-}
      \label{eq:lagmultNUpd}
      \\
       & t_{\varphi,i}^{k+1}
      = \tensorOne{m}_i \cdot P_{(0, c - \tan(\theta) t_{\varphi,N}^{k+1} )}
      \left(t_{\varphi,1}^{k} \tensorOne{m}_1 + t_{\varphi,2}^{k} \tensorOne{m}_2 + \varepsilon_{T} \Delta_n \tensorOne{g}_{T,\varphi} \right), \quad i \in \{1, 2\}.
      \label{eq:lagmultTUpd}
    \end{alignat}
    \label{eq:lagmultUpd}\null
  \end{subequations}
\end{linenomath}
\cref{alg:uzawaproc} summarizes the steps outlined above for solving the discrete system \eqref{eq:weakformDisc}.
\begin{algorithm}
  \caption{Uzawa Procedure}\label{alg:uzawaproc}
  \begin{algorithmic}[1]
    \State Initialize variables
    \While{ convergence }
    \While{$\|\vec{r}_u\|_2 \leq \tau $ }
    \State Compute $\vec{r}_u$ and $\frac{\partial \vec{r}_u}{\partial \vec{u}}={\Mat{A}}_{uu}$
    \State Compute $\delta \vec{u}$ solving \cref{eq:NR1}
    \State $\vec{u}_{\ell+1} = \vec{u}_{\ell} + \delta \vec{u}$
    \EndWhile
    \State Check for convergence
    \State Update $\tensorOne{t}_n^{h,k+1}$ using \cref{eq:lagmultUpd}
    \EndWhile
  \end{algorithmic}
\end{algorithm}
Note that at each update of the Lagrange multiplier, a system of nonlinear equations must be solved (\cref{eq:NR1}).
This can be particularly challenging during the first time step when the Lagrange multiplier may be significantly different from its true value.
In such cases, solving a nonlinear equation can be both expensive and ineffective since the value is only used temporarily.
To optimize and accelerate convergence, an alternative solving strategy can be employed.
Specifically, the Lagrange multiplier can be updated at each Newton iteration.
This alternative procedure is outlined in \cref{alg:alternativeAlgo}.
\begin{algorithm}
  \caption{Alternative procedure}\label{alg:alternativeAlgo}
  \begin{algorithmic}[1]
    \State Initialize variables
    \While{ convergence }
    \State Compute $\vec{r}_u$ and $\frac{\partial \vec{r}_u}{\partial \vec{u}}={\Mat{A}}_{uu}$
    \State Compute $\delta \vec{u}$ solving \cref{eq:NR1}
    \State $\vec{u}_{\ell+1} = \vec{u}_{\ell} + \delta \vec{u}$
    \State Check for convergence
    \State Update $\tensorOne{t}_n^{h,k+1}$ using \cref{eq:lagmultUpd}
    \EndWhile
  \end{algorithmic}
\end{algorithm}
This second procedure can be viewed as an Augmented Lagrangian approach with an adaptive tangential penalty coefficient, which depends on the directionality and the variation of the normal traction with respect to the normal jump.
However, in both cases, the linearization of the return mapping leads to a non-symmetric linear system when frictional contact is involved.

From the perspective of linear solvers, it is often more advantageous to work with symmetric linear systems.
By neglecting the contribution from the derivative of the tangential traction with respect to the normal displacement jump, we obtain a symmetric Jacobian system (see \ref{sec:appendixA}).
In other words, instead of implementing the return mapping in \cref{eq:lagmultTUpd}, the Coulomb condition is enforced based on a previously obtained estimate of the normal traction, i.e.
\begin{align}
   & t_{\varphi,i}^{k+1}
  = \tensorOne{m}_i \cdot P_{(0, c - \tan(\theta) t_{\varphi,N}^{k} )}
  \left(t_{\varphi,1}^{k} \tensorOne{m}_1 + t_{\varphi,2}^{k} \tensorOne{m}_2 + \varepsilon_{T} \Delta_n \tensorOne{g}_{T,\varphi} \right), \quad i \in \{1, 2\}.
  \label{eq:lagmultTUpdSym}\null
\end{align}

\subsection{Stability}\label{subsec:stability}
To analyze stability, we focus on the saddle-point problem associated with the stick condition, as it represents the most critical case for assessing the stability of the discretization.
In particular, as we approach convergence, i.e., $\tensorOne{t}_n^{h,k+1} \simeq \tensorOne{t}_n^{h,k} = \tensorOne{t}_n^{h}$, the discrete problem becomes:
\begin{linenomath}
  \begin{subequations}
    \begin{alignat}{2}
       & a(\tensorOne{v}^h, \tensorOne{u}_n^h)
      + b(\tensorOne{t}_n^{h}, \tensorOne{v}^h)
      = l_1(\tensorOne{v}^h),
       &                                        &
      \quad
      \forall \tensorOne{v}^h \in \boldsymbol{\mathcal{U}}^{h,1},
      \\
       & b(\tensorOne{\mu}^h,\tensorOne{u}_n^h)
      = l_2(\tensorOne{\mu}^h),
       &                                        &
      \quad
      \forall \tensorOne{\mu}^h \in \boldsymbol{\mathcal{M}}^h,
    \end{alignat}\label{eq:SPdisc}\null
  \end{subequations}
\end{linenomath}
where the bilinear forms ${a}(\tensorOne{v}^h, \tensorOne{u}_n^h)$ and $b(\tensorOne{\mu}^{h}, \tensorOne{u}_n^h)$ are given by:
\begin{linenomath}
  \begin{subequations}
    \begin{alignat}{2}
       & {a}(\tensorOne{v}^h, \tensorOne{u}_n^h)
      = (\tensorOne{v}^h, \tensorTwo{\sigma}(\tensorOne{u}_n^h))_{\Omega}
      + (\llbracket \tensorOne{v}^h \rrbracket, \tensorTwo{\varepsilon}^{-1} \cdot \llbracket \tensorOne{u}_n^{h} \rrbracket)_{\Gamma_f},
      \\
       & b(\tensorOne{\mu}^h, \tensorOne{u}_n^h)
      = (\tensorOne{\mu}^h,  \llbracket \tensorOne{u}_n^{h} \rrbracket)_{\Gamma_f}.
    \end{alignat}\label{eq:abForm}\null
  \end{subequations}
\end{linenomath}
The stability of problem \eqref{eq:SPdisc} is guaranteed if the following inf-sup condition holds:
\begin{equation}
  \inf_{\tensorOne{\mu}^h \in \boldsymbol{\mathcal{M}}^h } \sup_{\tensorOne{u}_n^h \in \boldsymbol{\mathcal{U}}^{h,1} } \frac{b(\tensorOne{\mu}^h, \tensorOne{u}_n^{h})}{\|{\tensorOne{u}_n^{h}\|_{\mathcal{U}} \|\tensorOne{\mu} }\|_{\mathcal{M}}} \geq \beta > 0.
  \label{eq:infsup}
\end{equation}
%
%
%
Although the spaces $\boldsymbol{\mathcal{U}}^{h,1}$ and $\boldsymbol{\mathcal{M}}^{h}$ are motivated by physical considerations and are commonly used in engineering applications, they do not, unfortunately, constitute an inf-sup stable pair for the discretization of problem \eqref{eq:weakformDisc}.
Moreover, the inf-sup condition remains critical even when employing iterative solution strategies such as the Uzawa procedure; see \cite{AroCasHamWhiTch24} for a discussion in the context of sequential solvers in poromechanics.
The next section presents and discusses a stabilization strategy to address these issues.

\begin{remark}
  It is important to mention that the strategy presenting in this work is not the only type of augmentation that can be applied.
  In \cite{BofBreFor13}, a general approach to augmenting a saddle point problem is discussed.
  Notably, depending on the type of augmentation used, the system can achieve stability and regain the coercivity of block (1,1).
  In our case, the specific augmentation enables us to restore the coercivity of block (1,1), but we still need to establish the stability of the saddle point problem.
\end{remark}

\section{Stabilization}
To satisfy the LBB condition \eqref{eq:infsup}, we introduce a well-known stabilization technique that involves enriching the piecewise linear continuous finite element space, $\boldsymbol{\mathcal{U}}^{h,1}$, with face bubble functions.
In the following section, we present the pair of spaces $(\boldsymbol{\mathcal{U}}^{h}, \boldsymbol{\mathcal{M}}^{h})$, where $\boldsymbol{\mathcal{U}}^{h,1} \subset \boldsymbol{\mathcal{U}}^{h}$.

\input{bubble_HEX_figs_final}
\input{bubble_TET_figs_final}
\input{bubble_WEDGE_figs_final}

\subsection{Stabilization by face bubble functions}
\label{sec:bubbles}
Let us begin by introducing the face bubble functions $b_{K,\varphi}$.
For each finite element $K$ that shares a face $\varphi$ with the contact surface $\Gamma_f$, the bubble functions are defined as follows:
\begin{itemize}
  \item
        $b_{K,\varphi} \in H^1(K)$ must have local support such that $\operatorname{supp}(b_{K,\varphi}) \subset K$;
  \item
        it is required that $\int_{\varphi} b_{\varphi} \; \mathrm{d} \Gamma \ne 0$ , where $b_{\varphi}$ denotes trace of $b_{K,\varphi}$ onto $\varphi$.
\end{itemize}
Several options can be considered, but the simplest form is to select a function that exhibits a quadratic form on the face $\varphi$ and vanishes on the other faces of the boundary $\partial K$ of the element $K$.
Specifically, on the reference element for each cell type considered in this work, the bubble function associated with face $\hat{\varphi}$ are defined as follows:
\begin{linenomath}
  \begin{subequations}
    \begin{alignat}{2}
      \text{reference hexahedron}:
       &
      \quad
      \hat{b}_{\hat{K}, \hat{\phi}}
      =
      \frac{1}{2} ( 1 \pm \hat{x}_j ) \prod_{i=1, i \ne j}^{3} (1 - \hat{x}_i^2),
      \label{eq:bubbleFunH}
      \\
      \text{reference tetrahedron}:
       &
      \quad
      \hat{b}_{\hat{K}, \hat{\phi}}
      =
      \prod_{i=1, i \ne j}^{4} \hat{\lambda}_{i,\hat{K}},
      \label{eq:bubbleFunT}
      \\
      \text{reference wedge}:
       &
      \quad
      \hat{b}_{\hat{K}, \hat{\phi}}
      =
      \begin{dcases}
        \frac{1}{2} ( 1 \pm \hat{x}_3 ) \prod_{i=1}^{2} \hat{\lambda}_{i,\hat{T}}
         &
        \text{(triangular faces)},
        \\
        ( 1 \pm \hat{x}_3^2 ) \prod_{i=1, i \ne j}^{2} \hat{\lambda}_{i,\hat{T}}
         &
        \text{(quadrilateral faces)},
      \end{dcases}
      \label{eq:bubbleFunW}
    \end{alignat}
  \end{subequations}
\end{linenomath}
where $\hat{x}_i$ are the Cartesian coordinates in the parent element system, and $\hat{\lambda}_{i,\hat{K}}$ and $\hat{\lambda}_{i,\hat{T}}$ denote the barycentric coordinates on the reference tetrahedron and on the triangular faces of the wedge reference element, respectively.
In \cref{eq:bubbleFunH}, $j$ denotes the index of the coordinate axis orthogonal to the face $\hat{\varphi}$.
In \cref{eq:bubbleFunT,eq:bubbleFunW}, $j$ is the index of the vertex opposite to the face $\varphi$ in $\hat{K}$ and $\hat{T}$, respectively.
\cref{fig:bubbles-hex,fig:bubbles-tet,fig:bubbles-wedge} illustrate example bubble functions both on the enriched face and within the reference element $\hat{K}$,  for hexahedral, tetrahedral, and wedge elements.

The stabilized discrete space $\boldsymbol{\mathcal{U}}^{h}$ is defined as follows:
\begin{alignat}{2}
   & \boldsymbol{\mathcal{U}}^{h}
  = \boldsymbol{\mathcal{U}}^{h,1}
  \oplus \boldsymbol{\mathcal{U}}^{b},
   &                              &
  \quad
  \boldsymbol{\mathcal{U}}^{b} = \operatorname{span}\left \{ [ b_{K,\varphi} ]^d  \right \}.
\end{alignat}
The set of basis functions for the space $\boldsymbol{\mathcal{U}}^{b}$ is denoted by  $\{ \tensorOne{\beta}_i \}$.
The degrees of freedom associated with $\boldsymbol{\mathcal{U}}^{h}$ are the values at the vertices of $\mathcal{T}$ and the total flux through of $\varphi$ of $(I-\Pi_1)\tensorOne{u}$, where $\Pi_1$ is the standard piecewise linear interpolant (Clement interpolant), $\Pi_1 : C(\bar{\Omega}) \rightarrow \boldsymbol{\mathcal{U}}^{h,1}$.
Then, the canonical interpolant, $\Pi_h: C(\bar{\Omega}) \rightarrow \boldsymbol{\mathcal{U}}^{h}$, is defined as:
\begin{alignat}{2}
   & \Pi_h \tensorOne{u}
  = \Pi_1 \tensorOne{u}
  + \sum_{\varphi \in \mathcal{F}_f }
  \left (
  b^{-}_{K, \varphi} \tensorOne{u}_{\varphi}^{-}  +  b^{+}_{K, \varphi} \tensorOne{u}_{\varphi}^{+}
  \right),
  \label{eq:inter_h}
\end{alignat}
with:
\begin{alignat}{2}
   & \tensorOne{u}_{\varphi}^{\pm}
  = \frac{1}{\int_{\varphi} b_{\varphi} \operatorname{d \Gamma} }  \int_{\Gamma} \gamma^{\pm}( \tensorOne{u} - \Pi_1 \tensorOne{u} )  \operatorname{d \Gamma}.
  \label{eq:defBubbleDof}
\end{alignat}
The enriched discrete space $\boldsymbol{\mathcal{U}}^{h}$ is endowed with the scaled broken Sobolev norm defined as follows:
\begin{alignat}{2}
   & \| u \|_U
  \equiv \normH{u}^2
  = \diam{\Omega}^{-2} \normL{u}^2 + \snormH{u}^2,
  \label{eq:normH_0}
\end{alignat}
where we denote $\|\cdot\|_{H^1(\Omega)} $ with $\normH{\cdot}$ to simplify the notation.
Before proceeding to the proof of inf-sup stability, let us state the following lemma, which presents an important inequality applicable to face bubble functions.
\begin{lemma}
  Let $\mathcal{T}$ be a shape-regular mesh with a conforming internal surface $\mathcal{F}_f$.
  Then, the following inequality holds for any $(K, \varphi) \in \mathcal{T} \times \mathcal{F}_f$ and for any $b^{\pm}_{K,\varphi}$ and for any $b_{\varphi}$ defined as above:
  \begin{alignat}{2}
     & \normH{b^{\pm}_{K,\varphi}}^2
    \le {C} {\diam{\Omega}^{-1}} \normLg{b_{\varphi}}^2,
    \label{eq:bound_bubble_H1_Lg_2}
  \end{alignat}
  with the constant $C$ depends on shape regularity and mesh element type.
\end{lemma}
\begin{proof}
  Let us begin the proof using \eqref{eq:normH_1} from \ref{sec:appendixB}.
  Inequality \eqref{eq:normH_1} holds for any $\tensorOne{u} \in \boldsymbol{\mathcal{U}}^{h} $, and consequently, it also holds for the bubble functions:
  \begin{alignat}{2}
     & \normH{b^{\pm}_{K,\varphi}}^2
    \le c \diam{\Omega}^{-2} \normL{b_{K,\varphi}}^2.
    \label{eq:normH_1_bubble}
  \end{alignat}
  Furthermore, for any parent element type, since the function $\hat{b}_{K,\varphi}$ has compact support in $K$, we can relate the $L_2$-norm over the element to the $L_2$-norm over the face as follows:
  \begin{align}
     & \normL{\hat{b}^{\pm}_{\hat{K},\hat{\varphi}}}
    = c \normLg{\hat{b}_{\hat{\varphi}}}.
    \label{eq:bound_bubble_L_Lg}
  \end{align}
  Using this result in conjunction with the inequalities \eqref{eq:affIneq} recalled in \ref{sec:appendixB}, which are obtained through affine mapping, we have:
  \begin{linenomath}
    \begin{subequations}
      \begin{alignat}{2}
        \normH{b^{\pm}_{K,\varphi}}^2 & \le c \diam{\Omega}^{-2} \normL{b^{\pm}_{K,\varphi}}^2                            \\
                                      & \le c \diam{\Omega}^{-2} \meas{K} \normL{\hat{b}^{\pm}_{\hat{K},\hat{\varphi}}}^2 \\
                                      & \le c \diam{\Omega}^{-2} \meas{K} \normL{\hat{b}_{\hat{\varphi}}}^2               \\
                                      & \le c \; {\meas{K}}{\meas{F}^{-1} \diam{\Omega}^{-2}} \normLg{b_{\varphi}}^2.
      \end{alignat}\label{eq:bound_bubble_H1_Lg_1}\null
    \end{subequations}
  \end{linenomath}
  Utilizing the assumptions of shape regularity and quasi-uniformity, specifically that
  \begin{alignat}{2}
     & {\meas{K}}{\meas{F}^{-1} \diam{\Omega}^{-1}}<c,
  \end{alignat}
  the relation can be rewritten as follows:
  \begin{alignat}{2}
     & \normH{b^{\pm}_{K,\varphi}}^2
    \le {C}{\diam{\Omega}^{-1}} \normLg{b_{\varphi}}^2.
    \label{eq:bound_bubble_H1_Lg_2_1}
  \end{alignat}
\end{proof}
The proof of inf-sup stability can be conducted by following the general framework established in \cite{BofBreFor13}.
In particular, the following theorem must be demonstrated:
\begin{theorem}
  If there exists a linear operator $\Pi_h: \boldsymbol{\mathcal{U}} \rightarrow \boldsymbol{\mathcal{U}}^{h}$  such that:
  \begin{linenomath}
    \begin{subequations}
      \begin{alignat}{2}
         & b (\tensorOne{\mu}, \Pi_h \tensorOne{u} - \tensorOne{u})
        = 0,
         &                                                          &
        \quad
        \forall \tensorOne{\mu} \in \boldsymbol{\mathcal{M}}^{h},
        \label{eq:thm_p1}                                             \\
         & \|\Pi_h \tensorOne{u} \|_U
        \leq C \| \tensorOne{u} \|_U,
         &                                                          &
        \quad
        \forall \tensorOne{u} \in \boldsymbol{\mathcal{U}}^{h},
        \label{eq:thm_p2}
      \end{alignat}\label{eq:theorem}
    \end{subequations}
  \end{linenomath}
  then the inf-sup condition \eqref{eq:infsup} holds true with $\beta = C^{-1}$.
\end{theorem}
\begin{proof}
  The inf-sup conditions can be proven using the so-called Fortin trick \cite{For77}.
  This approach involves rewriting the canonical interpolant $\Pi_h: C(\bar{\Omega}) \rightarrow \boldsymbol{\mathcal{U}}^{h}$ as follows:
  \begin{linenomath}
    \begin{subequations}
      \begin{alignat}{2}
         & \Pi_h \tensorOne{u}
        = \Pi_1 \tensorOne{u}
        + \Pi_2 (I - \Pi_1) \tensorOne{u},
      \end{alignat}\label{eq:FortinTrick}\null
    \end{subequations}
  \end{linenomath}
  that, in our case, using \cref{eq:inter_h}:
  \begin{alignat}{2}
     & \Pi_2 (I - \Pi_1) \tensorOne{u}
    = \sum_{\varphi \in \mathcal{F}_f } \left ( b^{-}_{K,\varphi} \tensorOne{u}_{\varphi}^{-}  +  b^{+}_{K, \varphi} \tensorOne{u}^{+}_{\varphi} \right),
    \label{eq:FortinTrick_P2}
  \end{alignat}
  Let us begin by proving the first part of the theorem, i.e., \cref{eq:thm_p1}.
  Since $\tensorOne{\mu} \in \boldsymbol{\mathcal{M}}^{h}$, we can assert that \cref{eq:thm_p1} is equivalent to:
  \begin{alignat}{2}
     & \int_{\varphi} \llbracket (\Pi_h \tensorOne{u} - \tensorOne{u} \rrbracket \operatorname{d \Gamma}
    = 0,
     &                                                                                                   &
    \quad
    \forall \varphi \in \mathcal{F}_f, \label{eq:bilinearform_b}
  \end{alignat}
  Substituting the definition of the interpolant given in \cref{eq:inter_h} and exploiting the linearity of the trace operator, we obtain:
  \begin{alignat}{2}
     & \int_{\varphi} b_{\varphi} \tensorOne{u}^{+}_{\varphi} - b_{\varphi}\tensorOne{u}^{-}_{\varphi}
    = \int_{\varphi} \llbracket (I - \Pi_1) \tensorOne{u} \rrbracket  d \Gamma
    = \int_{\varphi} \gamma^{+}(\tensorOne{u} - \Pi_1 \tensorOne{u}) d \Gamma -  \int_{\varphi} \gamma^{-}(\tensorOne{u} - \Pi_1 \tensorOne{u}) d \Gamma,
     &                                                                                                 &
    \quad
    \forall \varphi \in \mathcal{F}_f
    \label{eq:bilinearform_b_1}
  \end{alignat}
  This equation is satisfied if the following two conditions hold:
  \begin{linenomath}
    \begin{subequations}
      \begin{alignat}{2}
         & \tensorOne{u}^{+}_{\varphi}  \int_{\varphi} b_{\varphi}d \Gamma
        = \int_{\varphi} \gamma^{+}(\tensorOne{u} - \Pi_1 \tensorOne{u}) d \Gamma,
         &                                                                 &
        \quad
        \forall \varphi \in \mathcal{F}_f,
        \\
         & \tensorOne{u}^{-}_{\varphi} \int_{\varphi} b_{\varphi} d \Gamma
        = \int_{\varphi} \gamma^{-}(\tensorOne{u} - \Pi_1 \tensorOne{u}) d \Gamma,
         &                                                                 &
        \quad
        \forall \varphi \in \mathcal{F}_f.
      \end{alignat}\label{eq:bilinearform_b_2}\null
    \end{subequations}
  \end{linenomath}
  Using the definitions of $\tensorOne{u}^{+}_{\varphi}$ and $\tensorOne{u}^{-}_{\varphi}$, it is straightforward to verify that both of these conditions hold.
  To complete the proof, we need to demonstrate the second relation of the theorem.
  By again utilizing the Fortin trick, we can establish \cref{eq:thm_p2} by simply proving that:
  \begin{alignat}{2}
     & \normH{\Pi_2 (I - \Pi_1) \tensorOne{u}}
    \leq C \normH{\tensorOne{u}},
     &                                         &
    \quad
    \forall \tensorOne{u} \in \boldsymbol{\mathcal{U}}^{h}.
    \label{eq:boundPi2}
  \end{alignat}
  Let us begin the proof of this inequality.
  By using the definition of $\normH{\Pi_2 (I - \Pi_1) \tensorOne{u}}$ and applying the Cauchy–Schwarz inequality, we obtain:
  \begin{align}
    \normH{\Pi_2 (I - \Pi_1 ) \tensorOne{u}}^2
     & = \sum^{n_f}_{\varphi=1} \normH{b^{+}_{K,\varphi} \tensorOne{u}^{+}_{\varphi}}^2 +  \normH{b^{-}_{K,\varphi} \tensorOne{u}^{-}_{\varphi} }^2  \nonumber \\
     & \le \sum^{n_f}_{\varphi=1} \normH{b^{+}_{K,\varphi}}^2 |\tensorOne{u}_{\varphi}^{+}|_2^2
    + \normH{b^{-}_{K, \varphi}}^2 |\tensorOne{u}_{\varphi}^{-}|_2^2.
    \label{eq:boundPi2_1}
  \end{align}
  By substituting the definitions of $\tensorOne{u}_{\varphi}^{-}$ and $\tensorOne{u}_{\varphi}^{+}$ using \cref{eq:defBubbleDof}, we can proceed further:
  \begin{align}
    \normH{\Pi_2 (I - \Pi_1 ) \tensorOne{u}}^2 & \le c \sum^{n_f}_{\varphi=1} \frac{1}{\int_{\varphi} b_{\varphi} \operatorname{d \Gamma}} \left (  \normH{b^{+}_{K, \varphi}}^2 \left ( \int_{\varphi} \gamma^{+}(\tensorOne{u}-\Pi_1 \tensorOne{u} ) \operatorname{d \Gamma}\right )^2 + \normH{b^{-}_{K, \varphi}}^2 \left ( \int_{\varphi} \gamma^{-}(\tensorOne{u}-\Pi_1 \tensorOne{u}) \operatorname{d \Gamma} \right )^2 \right ).
    \label{eq:boundPi2_2}
  \end{align}
  Since the mesh is affine, the integral of the bubble function on the face $\varphi$ can be substituted by $\int_{\varphi} b_{\varphi}\operatorname{d\Gamma}= c \meas{\varphi}$.
  By exploiting Lemma, we have:
  \begin{align}
    \normH{\Pi_2 (I - \Pi_1 ) \tensorOne{u}}^2 & \le c \sum^{n_f}_{\varphi=1} \diam{\Omega}^{-1} \meas{\varphi}^{-1} \normLg{b_{\varphi}}^2 \left ( \normLg{\gamma^{+}(\tensorOne{u}-\Pi_1 \tensorOne{u})}^2 + \normLg{\gamma^{-}(\tensorOne{u}-\Pi_1 \tensorOne{u})}^2 \right ).
    \label{eq:boundPi2_3}
  \end{align}
  By utilizing the inequality \eqref{eq:gamAffIneq_v_hv} and the triangle inequality, we obtain the following expression:
  \begin{align}
    \normH{\Pi_2 (I - \Pi_1 ) \tensorOne{u}}^2
     & \le c \sum^{n_f}_{\varphi=1} \diam{\Omega}^{-1} \normLg{\gamma^{+}(\tensorOne{u}-\Pi_1 \tensorOne{u})}^2 + \normLg{\gamma^{-}(\tensorOne{u}-\Pi_1 \tensorOne{u})}^2 \nonumber
    \\
     & \le c \sum^{n_f}_{\varphi=1} \diam{\Omega}^{-1} \normLg{\gamma^{+}(\tensorOne{u})}^2 +\normLg{\gamma^{-}(\tensorOne{u})}^2
    \label{eq:boundPi2_4}
  \end{align}
  Finally, using continuity arguments, it is possible to prove that the $L_2$-norm of the trace is always less than or equal to the $H_1$-norm.
  Thus, we can express \cref{eq:boundPi2_4} as:
  \begin{align}
    \normH{\Pi_2 (I - \Pi_1 ) \tensorOne{u}}^2 & \le c \sum^{n_f}_{\varphi=1} \left (\diam{\Omega}^{-1} \normLg{\gamma^{+}(\tensorOne{u})}^2 + \diam{\Omega} \normLg{\gamma^{+} (\nabla \tensorOne{u})}^2 +  \right. \nonumber \\
                                               & \left . + \diam{\Omega}^{-1} \normLg{\gamma^{-}(\tensorOne{u})}^2 + \diam{\Omega} \normLg{\gamma^{-} (\nabla \tensorOne{u})}^2 \right ) \le C \normH{\tensorOne{u}}^2 .
    \label{eq:boundPi2_5}
  \end{align}
\end{proof}
The proof of the inf-sup condition has been established for bubble functions applied on both sides of the interface element.
It is also possible to stabilize the problem using one-sided bubble functions, as demonstrated in \cite{HauLeT07,DroEncFaiHaiMas24}.
However, numerical results presented in \cite{DroEncFaiHaiMas24} indicate that employing two-sided bubbles leads to a more stable formulation.
This increased stability, however, comes at the cost of introducing additional degrees of freedom into the system.
While the use of one-sided bubbles may simplify the implementation and reduce computational overhead, the benefit of enhanced stability from two-sided bubbles makes them a preferable choice in many scenarios.

\subsection{Algebraic Formulation of the Discrete System}
\label{sec:practical_impl}
By using the enriched discrete space $\boldsymbol{\mathcal{U}}^{h}$, we introduce additional degrees of freedom into the discrete algebraic system. In particular, \cref{eq:NR1,eq:NR2} become:
\begin{linenomath}
  \begin{subequations}
    \begin{alignat}{2}
       & \text{solve}
       &              &
      \quad
      \begin{bmatrix}
        {\Mat{A}}_{bb} & {\Mat{A}}_{bu} \\
        {\Mat{A}}_{ub} & {\Mat{A}}_{uu}
      \end{bmatrix}
      \begin{bmatrix}
        \delta \vec{u}_{b} \\
        \delta \vec{u}
      \end{bmatrix}
      = -
      \begin{bmatrix}
        \vec{r}_b \\
        \vec{r}_u
      \end{bmatrix},
      \label{eq:NR1bubbles}
      \\
       & \text{set}
       &              &
      \quad
      \begin{bmatrix}
        \vec{u}_{b} \\
        \vec{u}
      \end{bmatrix}_{\ell+1}
      =
      \begin{bmatrix}
        \vec{u}_{b} \\
        \vec{u}
      \end{bmatrix}_{\ell}
      +
      \begin{bmatrix}
        \delta \vec{u}_{b} \\
        \delta \vec{u}
      \end{bmatrix}
    \end{alignat}\label{eq:bubbleSystem}\null
  \end{subequations}
\end{linenomath}
%
Expressions for vector $\vec{r}_b$ and matrices ${\Mat{A}}_{bb}$, ${\Mat{A}}_{bu}$ and ${\Mat{A}}_{ub}$ are given in \ref{sec:appendixA}.
It is important to note that ${\Mat{A}}_{bb}$ has a block diagonal matrix representation.
Hence, we can eliminate the degrees of freedom corresponding to the bubble functions, resulting in a system that retains the same degrees of freedom as the original one, namely
\begin{align}
  \mathcal{A} & ={\Mat{A}}_{uu} - {\Mat{A}}_{ub} {\Mat{A}}^{-1}_{bb} {\Mat{A}}_{bu},
  \label{eq:hatA_staticCond}
\end{align}
The resulting scheme is equivalent to a stabilized low-order displacement constant piecewise traction discretization, in which stabilization terms appear in every sub-block.
Optimal order error estimates for this stabilized scheme can be derived from the estimates provided in \cite{Woh12} for the pair of spaces $(\boldsymbol{\mathcal{U}}^h, \boldsymbol{\mathcal{M}}^h)$.
The error estimate is obtained by utilizing these findings along with the standard error estimates for linear finite elements.

\begin{remark}
  It is also important to note that employing bubble functions may present a quadrature challenge for certain types of elements.
  It is necessary to integrate the terms of the formulation that involve bubble functions with sufficient accuracy to achieve optimal results.
  For tetrahedra and triangles, we must utilize multi-point quadrature rules to ensure exact integration of the contributions, employing a 14-point rule for tetrahedra and a 4-point rule for triangles, respectively.
\end{remark}

\section{Numerical Results}
In this section, we present numerical results to validate and assess the performance of the proposed formulation. The content is organized as follows:

\begin{enumerate}
  \item We first consider two theoretical benchmarks, where numerical solutions are compared against analytical ones. These test cases are used to validate the accuracy of the formulation.
  \item Next, we examine a series of problems designed to evaluate both linear and nonlinear convergence behavior, with particular attention to the influence of the penalty parameter.
  \item Finally, we present a field-scale simulation involving multiple faults to demonstrate the robustness and reliability of the method in scenarios representative of real-world subsurface applications.
\end{enumerate}

\subsection{Assessment of Numerical Accuracy and Convergence}
In this subsection, we use two benchmark problems to validate the implementation and verify the theoretical convergence behavior of the numerical error.
\subsubsection{Inclined fault subject to constant compression}
The first example involves a single crack within a two-dimensional infinite domain subjected to constant uniaxial compression (Figure \ref{fig:singleFault}).
\begin{figure}
  \centering
  \hfill
  \begin{subfigure}[b]{.45\linewidth}
    \centering
    \includegraphics[height=12em]{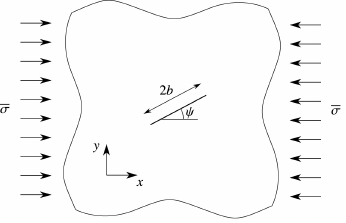}
    \caption{}
  \end{subfigure}
  \hfill
  \begin{subfigure}[b]{.45\linewidth}
    \small
    \centering
    \begingroup
    \renewcommand{\arraystretch}{1.4} 
    \begin{tabular}{lll}
      \toprule
      Quantity                                 & Value           & Unit \\
      \midrule
      Young's modulus ($E$)                    & $1 \times 10^5$ & [Pa] \\
      Poisson's ratio ($\nu$)                  & $0.4$           & [-]  \\
      Fracture inclination ($\Psi$)            & $20.0$          & [-]  \\
      Fracture length ($2b$)                   & $2.0$           & [m]  \\
      Friction angle ($\theta$)                & $30.0$          & [-]  \\
      Compressive stress ($\overline{\sigma}$) & $1.0$           & [Pa] \\
      \bottomrule
    \end{tabular}
    \endgroup
    \caption{}
  \end{subfigure}
  \hfill\null
  \caption{Inclined fault subject to constant compression: (a) domain sketch and (b) physical parameters.}\label{fig:singleFault}
\end{figure}
The benchmark is thoroughly detailed by \citep{PhaNapGraKap03}, and it serves as the basis for implementation validation also in \citep{FraCasWhiTch20, BeaChoLaaMas23}.
In this context, let $E$ and $\nu$ represent the linear elastic parameters, $\overline{\sigma}$ denote the magnitude of the compressive stress, $\alpha$ be the inclination angle of the fracture, $2b$ its length, and $\theta$ the friction angle according to Coulomb's criterion (with zero cohesion). The analytical solution yields the normal traction and the sliding on the fracture, as follows:
\begin{align}
  t_N                     & = -\overline{\sigma} \sin^2(\alpha),                                                                                                            \\
  \| \tensorOne{g}_T \|_2 & = \frac{4(1-\nu^2)}{E} \Bigl( \overline{\sigma} \sin(\alpha) \bigl(\cos(\alpha) - \sin(\alpha) \tan(\theta) \bigr) \Bigr) \sqrt{b^2-(b-\xi)^2},
\end{align}
respectively, where $\xi$ is the curvilinear abscissas along the fracture such that $0 \leq \xi \leq 2b$.
A comparison between the numerical and analytical solutions is presented for both the normal traction (Figure \ref{fig:singleFault_tn_b}) and slip displacement (Figure \ref{fig:singleFault_slip_b}). The computed displacement aligns closely with the expected values across the entire domain; however, significant discrepancies in the traction are observed near the fracture tip, where oscillations emerge. Even after refining the mesh, these oscillations persist but remain confined closer to the fracture tip, attributed to the singular behavior of the stress solution in that region.
Figures \ref{fig:singleFault_tn_a} and \ref{fig:singleFault_slip_a} illustrate the convergence properties, i.e., the error dependence on mesh size of the numerical formulation. Due to the oscillations in the traction approximation near the fracture tip, for this analysis, we focus as in \citep{FraCasWhiTch20, BeaChoLaaMas23} only on the central 90\% of the fracture length when calculating the $L_2$-norm of the normal traction error. The error convergence profile is slightly super-linear both for the traction and for the tangential slip, confirming the accuracy and validating the formulation used for tetrahedral, hexahedral and wedge elements.
\begin{figure}
  \small
  \begin{subfigure}[b]{.475\linewidth}
    \centering
    \begin{tikzpicture}
      \begin{axis}[
          width=\linewidth,
          height=0.75\linewidth,
          grid=both,
          minor x tick num = 4,
          minor y tick num = 4,
          major grid style={line width=0.4pt,draw=gray!50},
          minor grid style={line width=0.2pt,draw=gray!15},
          xmin=-0.1,xmax=2.1,ymin=0,ymax=15,
          xlabel={{ $y$ [m]} }, ylabel={ { $|t_N|$ [MPa] }},
          ylabel near ticks,xlabel near ticks,
          scaled y ticks=false,
          legend style={at={(0.5,0.05)},anchor=south},
          set layers]
        \addplot [mark=*, mark size=0.8pt, solid, black, on layer=axis background] table [x=xsi,y=tn] {./20comp_ytn.txt};
        \addplot [mark=none, solid, line width=1pt, red,  on layer=axis foreground] table [x=xsi,y=tn]  {./20comp_ytnAnalyticalSolution.txt};
        \legend{{Numerical Solution},{Analitical Solution}}
      \end{axis}
    \end{tikzpicture}
    \caption{}
    \label{fig:singleFault_tn_b}
  \end{subfigure}
  \hfill
  \begin{subfigure}[b]{.475\linewidth}
    \centering
    \begin{tikzpicture}
      \begin{axis}[
          width=\linewidth,
          height=0.75\linewidth,
          grid=both,
          minor x tick num = 4,
          minor y tick num = 1,
          major grid style={line width=0.4pt,draw=gray!50},
          minor grid style={line width=0.2pt,draw=gray!15},
          xmin=-0.1,xmax=2.1,ymin=0,ymax=0.007,
          xlabel={{ $y$ [m]} }, ylabel={ { $\|\tensorOne{g}_T\|_2$ [m] }},
          ylabel near ticks,xlabel near ticks,
          scaled y ticks=false,
          legend style={at={(0.5,0.05)},anchor=south},
          set layers]
        \addplot [mark=*, mark size=0.8pt, solid, black, on layer=axis background] table [x=xsi,y=slip] {./20comp_ytn.txt};
        \addplot [mark=none, solid, line width=1pt, red,  on layer=axis foreground] table [x=xsi,y=slip]  {./20comp_xsiSlipAnalyticalSolution.txt};
        \legend{{Numerical Solution},{Analitical Solution}}
      \end{axis}
    \end{tikzpicture}
    \caption{}
    \label{fig:singleFault_slip_b}
  \end{subfigure}
  \caption{Inclined fault subject to constant compression: Comparison between numerical and analytical (a) normal traction $|t_N|$ and (b) tangential slip $\| \tensorOne{g}_T \|_2$ along the fracture.}\label{fig:singleFault_tn}
\end{figure}
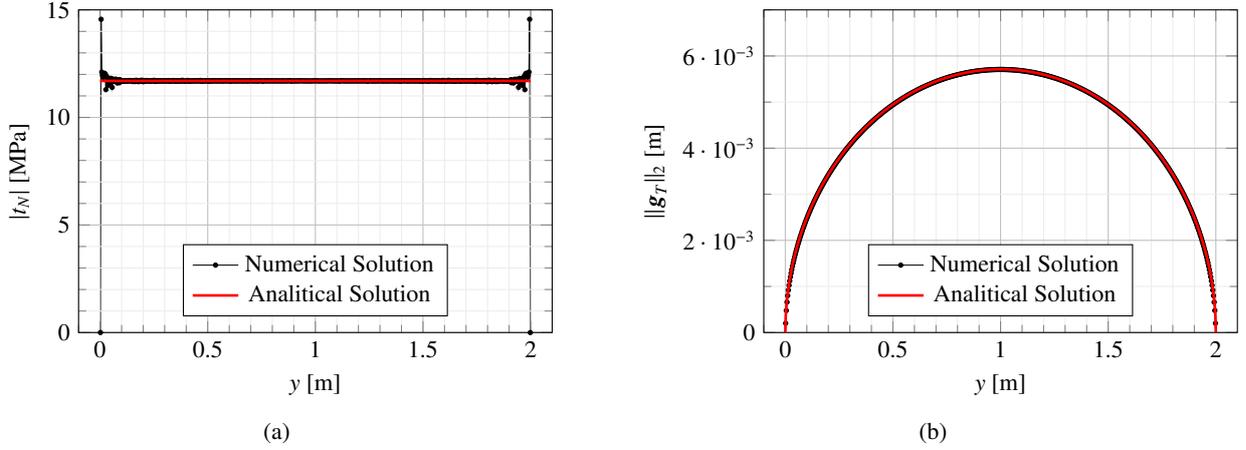
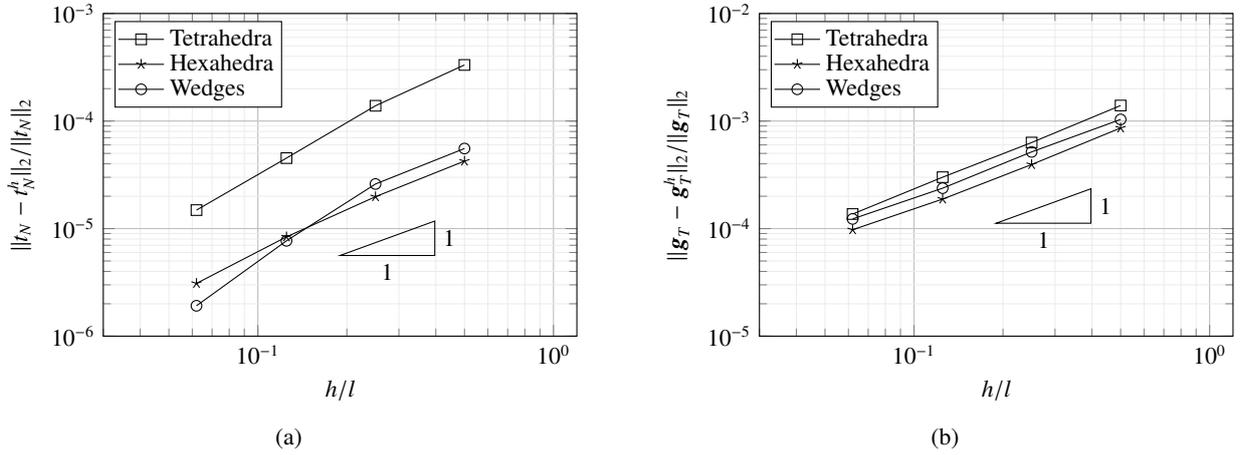
\begin{figure}
  \centering
  \small
  \begin{subfigure}[b]{.475\linewidth}
    \centering
    \begin{tikzpicture}
      \begin{loglogaxis}[ 
          width=\linewidth,
          height=0.75\linewidth,
          grid=both,
          major grid style={line width=0.4pt,draw=gray!50},
          minor grid style={line width=0.2pt,draw=gray!15},
          xmin=3.e-2,xmax=1.2,
          ymin=1.e-6,ymax=1.e-3,
          xlabel={ $h/l$ }, ylabel={  $\|t_N-t_N^h\|_2/\|t_N\|_2$  },
          ylabel near ticks,xlabel near ticks,
          legend style={
              at={(0.025,0.975)},
              anchor=north west,
              cells={anchor=west},
              inner sep=1pt,
              outer sep=0pt,
              row sep=-2pt
            }
        ]
        \addplot [black,mark=square] table [x=h,y=err_T] {./error_20comp50.txt};
        \addplot [black,mark=star] table [x=h,y=err_H] {./error_20comp50.txt};
        \addplot [black,mark=o] table [x=h,y=err_W] {./error_20comp50.txt};
        \logLogSlopeTriangle{0.7}{0.2}{0.25}{1}{black};
        \legend{{Tetrahedra},{Hexahedra},{Wedges}}
      \end{loglogaxis}
    \end{tikzpicture}
    \caption{}
    \label{fig:singleFault_tn_a}
  \end{subfigure}
  \hfill
  \begin{subfigure}[b]{.475\linewidth}
    \centering
    \begin{tikzpicture}
      \begin{loglogaxis}[
          width=\linewidth,
          height=0.75\linewidth,
          grid=both,
          major grid style={line width=0.4pt,draw=gray!50},
          minor grid style={line width=0.2pt,draw=gray!15},
          xmin=3.e-2,xmax=1.2,
          ymin=1.e-5,ymax=1.e-2,
          xlabel={{$h/l$} }, ylabel={ {$\|\tensorOne{g}_T-\tensorOne{g}_T^h\|_2/\|\tensorOne{g}_T\|_2$}  },
          ylabel near ticks,xlabel near ticks,
          legend style={
              at={(0.025,0.975)},
              anchor=north west,
              cells={anchor=west},
              inner sep=1pt,
              outer sep=0pt,
              row sep=-2pt
            }
        ]
        \addplot [black,mark=square] table [x=h,y=err_T_j] {./error_20comp50.txt};
        \addplot [black,mark=star] table [x=h,y=err_H_j] {./error_20comp50.txt};
        \addplot [black,mark=o] table [x=h,y=err_W_j] {./error_20comp50.txt};
        \logLogSlopeTriangle{0.7}{0.2}{0.35}{1}{black};
        \legend{{Tetrahedra},{Hexahedra},{Wedges}}
      \end{loglogaxis}
    \end{tikzpicture}
    \caption{}
    \label{fig:singleFault_slip_a}
  \end{subfigure}
  \caption{Inclined fault subject to constant compression: Convergence of the relative $L_2$-error in (a) normal traction and  (b) tangential slip.}\label{fig:singleFault_slip}
\end{figure}
\subsubsection{Dislocated reservoir with a vertical fault}
The second test case is described in \cite{NovShoVosHajJan2023}.
The domain approximates an infinite horizontal dislocated reservoir in $x$ and $z$ directions with a vertical fault at the center of the reservoir. Figure \ref{fig:verticalFault} shows a sketch of the domain and provides the value used for the geometrical setup and the physical parameters.
\begin{figure}
  \centering
  \hfill
  \begin{subfigure}[b]{.45\linewidth}
    \centering
    \includegraphics[height=17em]{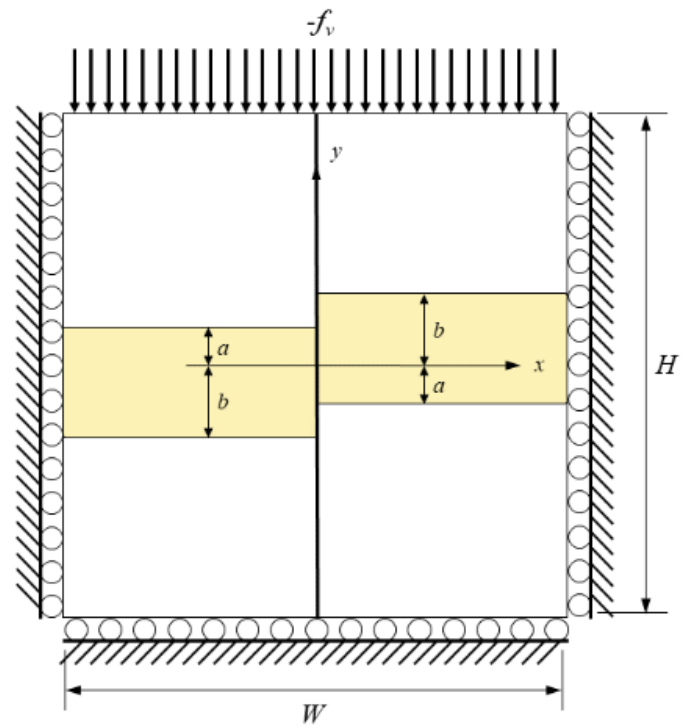}
    \caption{}\label{fig:verticalFault_a}
  \end{subfigure}
  \hfill
  \begin{subfigure}[b]{.45\linewidth}
    \small
    \centering
    \begingroup
    \renewcommand{\arraystretch}{1.4} 
    \begin{tabular}{lll}
      \toprule
      Quantity                             & Value  & Unit  \\
      \midrule
      Reservoir geometry ($a$)             & $75$   & [m]   \\
      Reservoir geometry ($b$)             & $150$  & [m]   \\
      Height of simulation domain ($H$)    & $4500$ & [m]   \\
      Width of simulation domain  ($W$)    & $4500$ & [m]   \\
      Shear modulus ($G$)                  & $6500$ & [MPa] \\
      Poisson's ratio ($\nu$)              & $0.15$ & [-]   \\
      Friction angle ($\theta$)            & $30.0$ & [-]   \\
      Incremental reservoir pressure ($p$) & $-25$  & [MPa] \\
      Biot's coefficient ($\alpha$)        & $0.9$  & [-]   \\
      \bottomrule
    \end{tabular}
    \endgroup
    \caption{}\label{fig:verticalFault_b}
  \end{subfigure}
  \hfill\null
  \caption{Dislocated reservoir with a vertical fault: (a) domain sketch and (b) physical parameters.}\label{fig:verticalFault}
\end{figure}
We have considered here two scenarios: the pre-slip case and the slip case.
In the pre-slip case, we focus on the tangential stress, whose analytical expression is given by the following relation:
\begin{align}
  \| \tensorOne{t}_T \|_2 & = \frac{C}{2} \ln \frac{(y - a)^2 (y + a)^2}{(y - b)^2(y + b)^2},
\end{align}
where $C$ is given by:
\begin{align}
  C & = \frac{(1 - 2 \nu ) \alpha p}{2 \pi (1 - \nu)}
\end{align}
In the second scenario, the slip begin and it is given by the following relation:
\begin{align}
  \| \tensorOne{g}_T \|_2 & = \frac{C}{A}
  \begin{cases}
    0 \quad        & \mbox{if} \quad  y \le -b,     \\
    -(y + b) \quad & \mbox{if} \quad -b < y \le -a, \\
    (a - b) \quad  & \mbox{if} \quad -a < y < a,    \\
    (y - b) \quad  & \mbox{if} \quad  a \le y < b,  \\
    0 \quad        & \mbox{if} \quad  b \le y
  \end{cases}
\end{align}
where:
\begin{align}
  A = G 2 \pi ( 1 - \nu).
\end{align}
Comparison between the analytical and numerical solution for the tangential traction is showed in Figure \ref{fig:verticalFault_tt_b}, while the comparison
between the analytical and numerical solution for
the slip along the fault is shown in Figure \ref{fig:verticalFault_slip_b}.
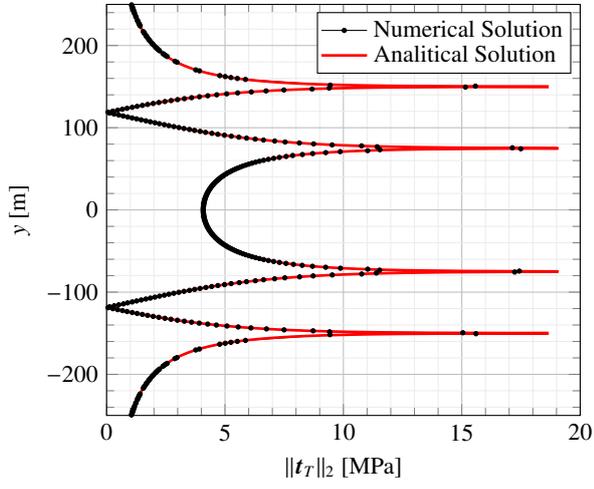
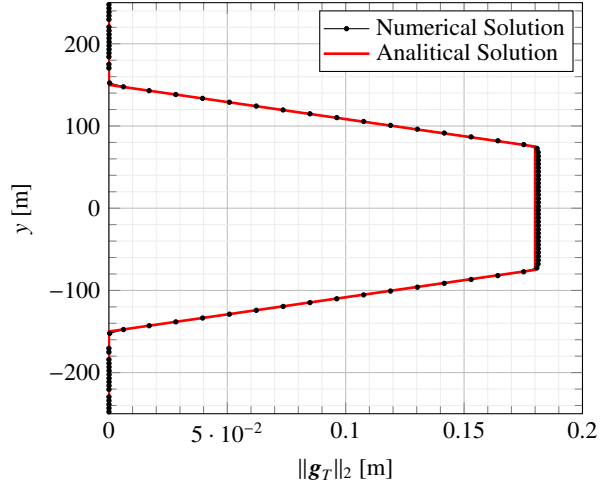
\begin{figure}
  \centering
  \small
  \begin{subfigure}[b]{.475\linewidth}
    \centering
    \begin{tikzpicture}
      \pgfplotsset{scaled y ticks=false}
      \begin{axis}[ 
          width=1.0\linewidth,height=0.9\linewidth,
          grid=both,
          minor x tick num = 4,
          minor y tick num = 4,
          major grid style={line width=0.4pt,draw=gray!50},
          minor grid style={line width=0.2pt,draw=gray!15},
          xmin=0,xmax=20,ymin=-250,ymax=250,
          ylabel={ $y$ [m] }, xlabel={ { $\| \tensorOne{t}_T\|_2$ [MPa] }},
          xtick distance=5,
          ytick distance=100,
          ylabel near ticks,xlabel near ticks,
          legend style={
              anchor=north east,
              cells={anchor=west},
              inner sep=1pt,
              outer sep=0pt,
              row sep=-2pt
            }]
        \addplot [mark=*, mark size=0.8pt, solid, black, on layer=axis background] table [x=tt,y=y] {./verFault_ytt.txt};
        \addplot [mark=none, solid, line width=1pt, red,  on layer=axis foreground] table [x=tt,y=y] {./verFault_yttAnalyticalSolution.txt};
        \legend{{Numerical Solution},{Analitical Solution}}
      \end{axis}
    \end{tikzpicture}
    \caption{}
    \label{fig:verticalFault_tt_b}
  \end{subfigure}
  \hfill
  \begin{subfigure}[b]{.475\linewidth}
    \centering
    \begin{tikzpicture}
      \pgfplotsset{scaled y ticks=false}
      \begin{axis}[ 
          width=1.0\linewidth,height=0.9\linewidth,
          grid=both,
          minor x tick num = 4,
          minor y tick num = 4,
          major grid style={line width=0.4pt,draw=gray!50},
          minor grid style={line width=0.2pt,draw=gray!15},
          xmin=0,xmax=0.20,ymin=-250,ymax=250,
          ylabel={{ $y$ [m]} }, xlabel={ { $\|\tensorOne{g}_T\|_2$ [m] }},
          xtick distance=0.05,
          ytick distance=100,
          ylabel near ticks,xlabel near ticks,
          legend style={
              anchor=north east,
              cells={anchor=west},
              inner sep=1pt,
              outer sep=0pt,
              row sep=-2pt
            }]
        \addplot [mark=*, mark size=0.8pt, solid, black, on layer=axis background] table [x=slip,y=y] {./verFault_yslip.txt};
        \addplot [mark=none, solid, line width=1pt, red,  on layer=axis foreground] table [x=slip,y=y] {./verFault_yslipAnalyticalSolution.txt};
        \legend{{Numerical Solution},{Analitical Solution}}
      \end{axis}
    \end{tikzpicture}
    \caption{}
    \label{fig:verticalFault_slip_b}
  \end{subfigure}
  \caption{Dislocated reservoir with a vertical fault: Comparison between numerical and analytical (a) tangential traction $\|\tensorOne{t}_T\|_2$ and (b) tangential slip $\| \tensorOne{g}_T \|_2$ along the fracture.
  }
  \label{fig:verticalFault_tt}
\end{figure}
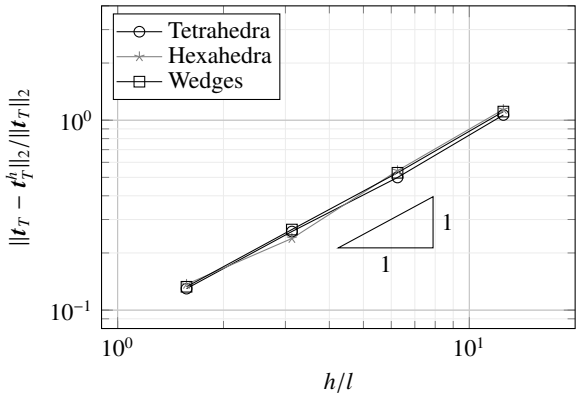
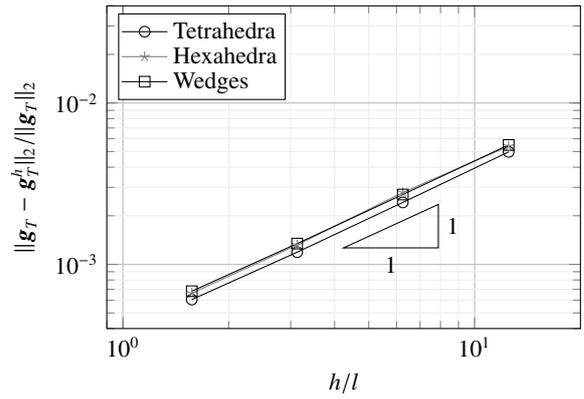
\begin{figure}
  \centering
  \small
  \begin{subfigure}[b]{.475\linewidth}
    \centering
    \begin{tikzpicture}
      \begin{loglogaxis}[ 
          width=\linewidth,
          height=0.75\linewidth,
          grid=both,
          major grid style={line width=0.4pt,draw=gray!50},
          minor grid style={line width=0.2pt,draw=gray!15},
          xmin=0.9e0,xmax=2e1,
          ymin=0.8e-1,ymax=0.4e1,
          xlabel={ $h/l$ },
          ylabel={ $\|\tensorOne{t}_T-\tensorOne{t}_T^h\|_2/\|\tensorOne{t}_T\|_2$ },
          ylabel near ticks,xlabel near ticks,
          legend style={
              at={(0.025,0.975)},
              anchor=north west,
              cells={anchor=west},
              inner sep=1pt,
              outer sep=0pt,
              row sep=-2pt
            }
        ]
        \addplot [black,mark=o] table [x=h,y=err_T] {./error_verFault.txt};
        \addplot [gray,mark=star] table [x=h,y=err_H] {./error_verFault.txt};
        \addplot [black,mark=square] table [x=h,y=err_W] {./error_verFault.txt};
        \logLogSlopeTriangle{0.7}{0.2}{0.25}{1}{black};
        \legend{{Tetrahedra},{Hexahedra},{Wedges}}
      \end{loglogaxis}
    \end{tikzpicture}
    \caption{}
    \label{fig:verticalFault_tt_a}
  \end{subfigure}
  \hfill
  \begin{subfigure}[b]{.475\linewidth}
    \centering
    \begin{tikzpicture}
      \begin{loglogaxis}[ 
          width=\linewidth,
          height=0.75\linewidth,
          grid=both,
          major grid style={line width=0.4pt,draw=gray!50},
          minor grid style={line width=0.2pt,draw=gray!15},
          xmin=0.9e0,xmax=2e1,
          ymin=0.4e-3,ymax=0.4e-1,
          xlabel={ $h/l$ },
          ylabel={ $\|\tensorOne{g}_T-\tensorOne{g}_T^h\|_2/\|\tensorOne{g}_T\|_2$  },
          ylabel near ticks,xlabel near ticks,
          legend style={
              at={(0.025,0.975)},
              anchor=north west,
              cells={anchor=west},
              inner sep=1pt,
              outer sep=0pt,
              row sep=-2pt
            }
        ]
        \addplot [black,mark=o] table [x=h,y=err_T_j] {./error_verFault.txt};
        \addplot [gray,mark=star] table [x=h,y=err_H_j] {./error_verFault.txt};
        \addplot [black,mark=square] table [x=h,y=err_W_j] {./error_verFault.txt};
        \logLogSlopeTriangle{0.7}{0.2}{0.25}{1}{black};
        \legend{{Tetrahedra},{Hexahedra},{Wedges}}
      \end{loglogaxis}
    \end{tikzpicture}
    \caption{}
    \label{fig:verticalFault_slip_a}
  \end{subfigure}
  \caption{Dislocated reservoir with a vertical fault: Convergence of the relative $L_2$-error in (a) tangential traction and (b) tangential slip.}
  \label{fig:verticalFault_slip}
\end{figure}
The error convergence profiles are provided in Figures \ref{fig:verticalFault_tt_a} and \ref{fig:verticalFault_slip_a}.
The profiles are linear both for the tangential traction and for the slip, confirming once again the accuracy and validating the formulation used for tetrahedral, hexahedral and wedge elements.

\subsection{Convergence analysis of nonlinear and linear solvers}
In this section, we examine three test cases designed to investigate nonlinear behavior and linear convergence. For this purpose, we have selected the following test cases:
\begin{itemize}
  \item {\it Stick-Slip-Open Case}: Two adjacent blocks are subjected to compression and shear, causing them to progressively slide against each other before separating (Figure \ref{fig:SSO_dom}).
  \item {\it Constant Slip Solution}: An analytical benchmark that simulates constant sliding between two blocks (Figure \ref{fig:CS_dom}).
  \item {\it T-Crack Case}: A 3D domain featuring a T-shaped fracture that is subject to pressurization (Figure \ref{fig:TF_dom}).
\end{itemize}
Each case in this subsection is selected to demonstrate the effectiveness of the Augmented Lagrangian Method (ALM) in specific scenarios. The Stick-Slip-Open case is used to analyze convergence when a state transition occurs across multiple elements. The Constant Slip Solution case investigates convergence when the slip condition completely alters the normal stress. In this example, as we will demonstrate in the following sections, there is a complete redistribution of normal stress to ensure that the equilibrium equation is satisfied. Finally, the T-Frac case illustrates how the formulation behaves in the presence of fault intersections.
For all these cases, we conduct a study on linear and nonlinear convergence, reporting the total number of linear and nonlinear iterations. Regarding the linear solver, we use the GMRES \cite{SaaSch86} Krylov subspace method preconditioned with Boomer AMG, powered by Hypre \cite{FalUlr02}, to solve the linear system in \cref{eq:NR1bubbles}. The preconditioner is computed using the separate displacement component technique introduced in~\cite{AxeGus78}.
The authors acknowledge that this may not be the optimal technique for preconditioning the system in \cref{eq:NR1bubbles}, but the development of ad hoc preconditioners for the augmented system with bubble degrees of freedom is beyond the scope of this paper. The primary goal of this comparison is to analyze how both nonlinear and linear convergence are affected by the type of problem, such as geometry and boundary conditions, and by the penalty parameters.
Moreover, this approach is compared here with the Lagrange Multiplier Method combined with an Active Set strategy \cite{FraCasWhiTch20} as implemented in GEOS \cite{Geos24}.
\begin{figure}
  \centering
  \hfill
  \begin{subfigure}[b]{.3\linewidth}
    \centering
    \includegraphics[height=12em]{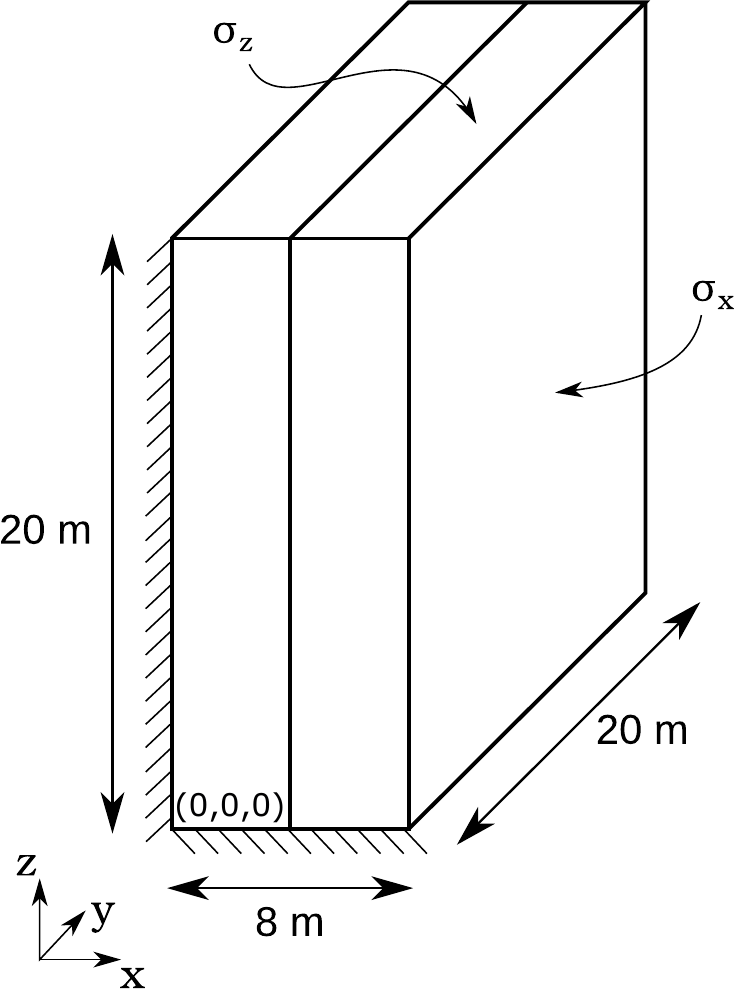}
    \caption{}
    \label{fig:SSO_dom}
  \end{subfigure}
  \begin{subfigure}[b]{.3\linewidth}
    \centering
    \includegraphics[height=12em]{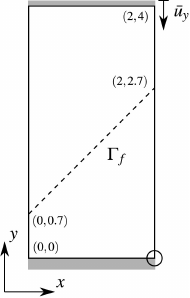}
    \caption{}
    \label{fig:CS_dom}
  \end{subfigure}
  \begin{subfigure}[b]{.3\linewidth}
    \centering
    \includegraphics[height=12em]{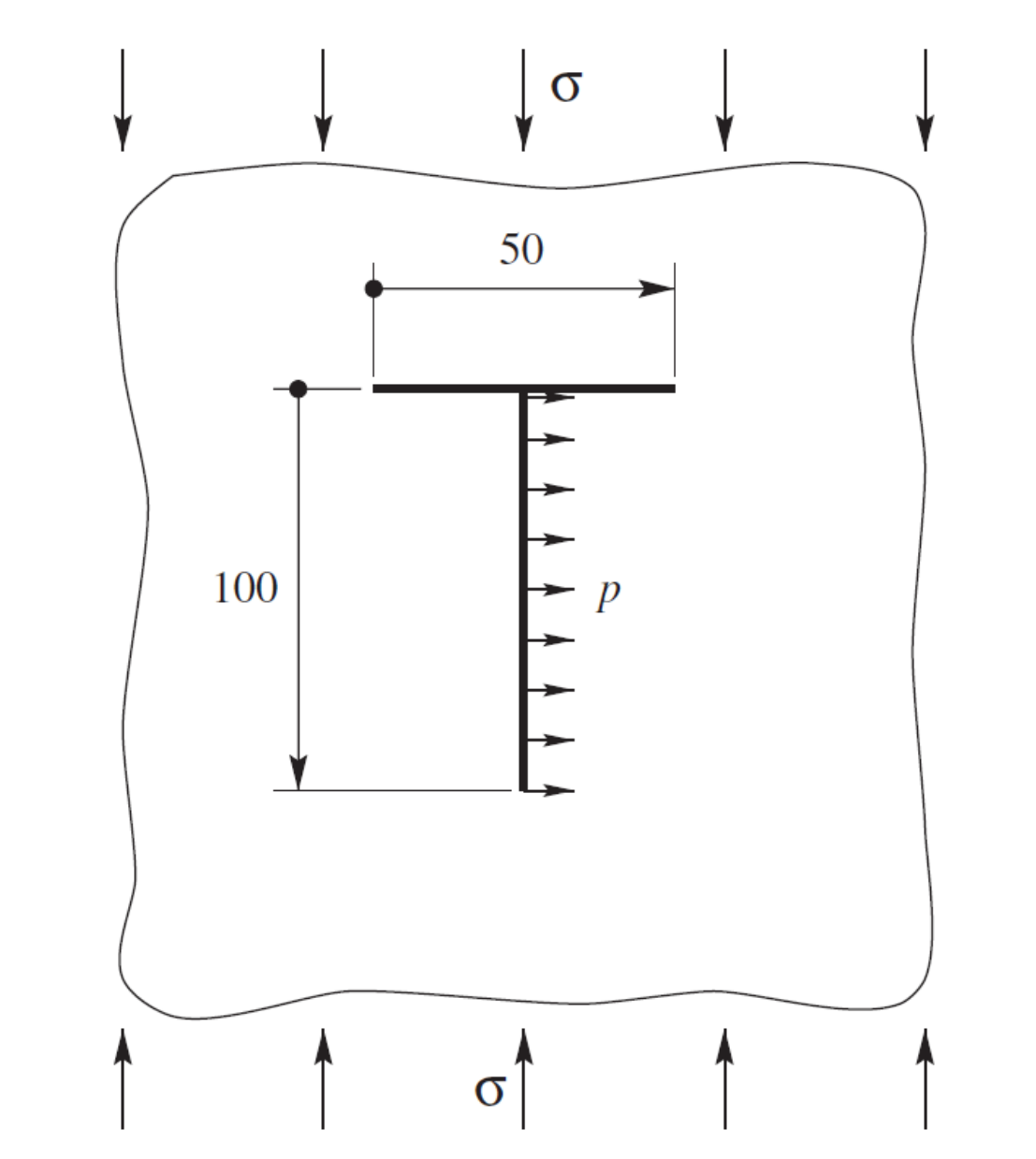}
    \caption{}
    \label{fig:TF_dom}
  \end{subfigure}
  \caption{Domain sketch for (a) Stick-Slip-Open Case, (b) Constant Slip Solution, and (c) T-Crack Case.}
\end{figure}

\subsubsection{Stick-Slip-Open Case}
The problem is illustrated in Figure \ref{fig:SSO_dom}. A prismatic elastic body is divided into two blocks by a vertical crack. The prism has a rectangular base measuring 8 m by 20 m and a height of 20 m. The domain is discretized using a structured mesh consisting of $4\times10\times10$ hexahedral elements.
The boundary conditions and loads are prescribed as shown in Figure \ref{fig:SSO_dom}.
Specifically, normal tractions $\sigma_x$ and $\sigma_z$ are applied to the top ($z$ = 20 m) and the right ($x$ = 8 m) face of the rightmost block.
The linear elastic material has a Young's modulus of $E=450$ MPa and a Poisson's ratio of $\nu=0.3$. An initial compressive stress state along the x-direction is assumed. The crack exhibits zero cohesion and a friction angle of $\theta=30^o$.
\input{figure12}
During the simulation, the tractions $\sigma_x$ and $\sigma_z$ vary as shown in Figure \ref{fig:SSO_sigma_history}, which also includes snapshots of the fracture state at selected time steps.
We consider a ten-step loading sequence.
Figures \ref{fig:SSO} show the cumulative total number of Uzawa, Newton and GMRES iterations over the entire simulation. In these two plots, we compare the two solving strategies described in section \ref{subsec:solutionStrategy}, as a function of the penalty parameter $\varepsilon_N$, which is set to the default value of 10$\bar{E}/h$.
Here, $\bar{E}$ represents the average Young's Modulus of the neighboring cells, and $h$ denotes the characteristic size of these cells.
It can be observed that Algorithm \ref{alg:alternativeAlgo} surpasses Algorithm \ref{alg:uzawaproc} in terms of performance, converging faster when the penalty parameter is around its default value. Notably, the estimate of 10$\bar{E}/h$ is close to the optimal value for many cases. However, fine-tuning this parameter specific to the problem at hand can yield further improvements in overall performance. Therefore, the choice of the iterative penalty parameter is crucial to ensuring the effectiveness of the method.
Indeed, larger values of the penalty parameter lead to faster convergence of the nonlinear loop. However, setting the penalty parameter too high can exacerbate the ill-conditioning of the linear system, which increases the number of iterations required by the linear solver. This can reintroduce the issues typically associated with penalty methods.
Figure \ref{fig:SSO_bar} illustrates the number of iterations obtained using the default setup for all the algorithms described in \cref{subsec:solutionStrategy}. It also includes, for comparison, the results obtained using a Lagrange Multiplier combined with an Active Set method \cite{FraCasWhiTch20}.
All the Augmented Lagrangian algorithms outperform the Lagrange multiplier method in this case. Additionally, the reader should keep in mind that the cost of a single iteration using the Lagrange Multiplier method is slightly higher, as it requires solving a Saddle Point problem.
Moreover, the symmetric version of \cref{alg:uzawaproc,alg:alternativeAlgo} can leverage the presence of a symmetric positive definite (SPD) system following the linearization, which provides several advantages from numerical point of view.
\begin{figure}
  \small
  \begin{subfigure}[b]{.33\linewidth}
    \centering
    \begin{tikzpicture}
      \pgfplotsset{scaled y ticks=false}
      \begin{loglogaxis}[ 
        width=\linewidth,height=1.25\linewidth,
        grid=both,
        major grid style={line width=0.4pt,draw=gray!50},
        minor grid style={line width=0.2pt,draw=gray!15},
        xmin=0.05,xmax=200.0,ymin=10,ymax=1000.0,
        xlabel={{ ${\varepsilon h}/{\overline{E}}$ } },
        ylabel={Uzawa iteration count},
        ylabel near ticks,xlabel near ticks,
        legend style={
            anchor=north east,
            cells={anchor=west},
            inner sep=1pt,
            outer sep=0pt,
            row sep=-2pt
          }]
        \addplot [black,mark=square] table [x=penalty,y=nUz1] {./solids-stick-slip-open.txt};
      \end{loglogaxis}
    \end{tikzpicture}
    \caption{}
  \end{subfigure}
  \hfill
  \begin{subfigure}[b]{.33\linewidth}
    \centering
    \begin{tikzpicture}
      \pgfplotsset{scaled y ticks=false}
      \begin{loglogaxis}[ 
        width=\linewidth,height=1.25\linewidth,
        grid=both,
        major grid style={line width=0.4pt,draw=gray!50},
        minor grid style={line width=0.2pt,draw=gray!15},
        xmin=0.05,xmax=200.0,ymin=10,ymax=1000.0,
        xlabel={{ ${\varepsilon h}/{\overline{E}}$ } },
        ylabel={Newton iteration count},
        ylabel near ticks,xlabel near ticks,
        legend style={
            anchor=north east,
            cells={anchor=west},
            inner sep=1pt,
            outer sep=0pt,
            row sep=-2pt
          }]
        \addplot [black,mark=*] table [x=penalty,y=nNL1] {./solids-stick-slip-open.txt};
        \addplot [black,mark=o] table [x=penalty,y=nNL2] {./solids-stick-slip-open.txt};
        \legend{{\cref{alg:uzawaproc}},{\cref{alg:alternativeAlgo}}}
      \end{loglogaxis}
    \end{tikzpicture}
    \caption{}
  \end{subfigure}
  \hfill
  \begin{subfigure}[b]{.33\linewidth}
    \centering
    \begin{tikzpicture}
      \pgfplotsset{scaled y ticks=false}
      \begin{loglogaxis}[  
        width=\linewidth,height=1.25\linewidth,
        grid=both,
        major grid style={line width=0.4pt,draw=gray!50},
        minor grid style={line width=0.2pt,draw=gray!15},
        xmin=0.05,xmax=200.0,ymin=800,ymax=10000.0,
        xlabel={{ ${\varepsilon h}/{\overline{E}}$ } },
        ylabel={GMRES iteration count},
        ylabel near ticks,xlabel near ticks,
        legend style={
            at={(0.5, 0.975)},
            anchor=north,
            cells={anchor=west},
            inner sep=1pt,
            outer sep=0pt,
            row sep=-2pt
          }]
        \addplot [black,mark=*] table [x=penalty,y=nL1] {./solids-stick-slip-open.txt};
        \addplot [black,mark=o] table [x=penalty,y=nL2] {./solids-stick-slip-open.txt};
        \legend{{\cref{alg:uzawaproc}},{\cref{alg:alternativeAlgo}}}
      \end{loglogaxis}
    \end{tikzpicture}
    \caption{}
  \end{subfigure}
  \caption{Stick-Slip-Open Case: Number of Uzawa (a), Newton (b), and GMRES (c) iterations as a function of the penalty parameter, with its default value equal to $10 \bar{E} / h$.}\label{fig:SSO}
\end{figure}
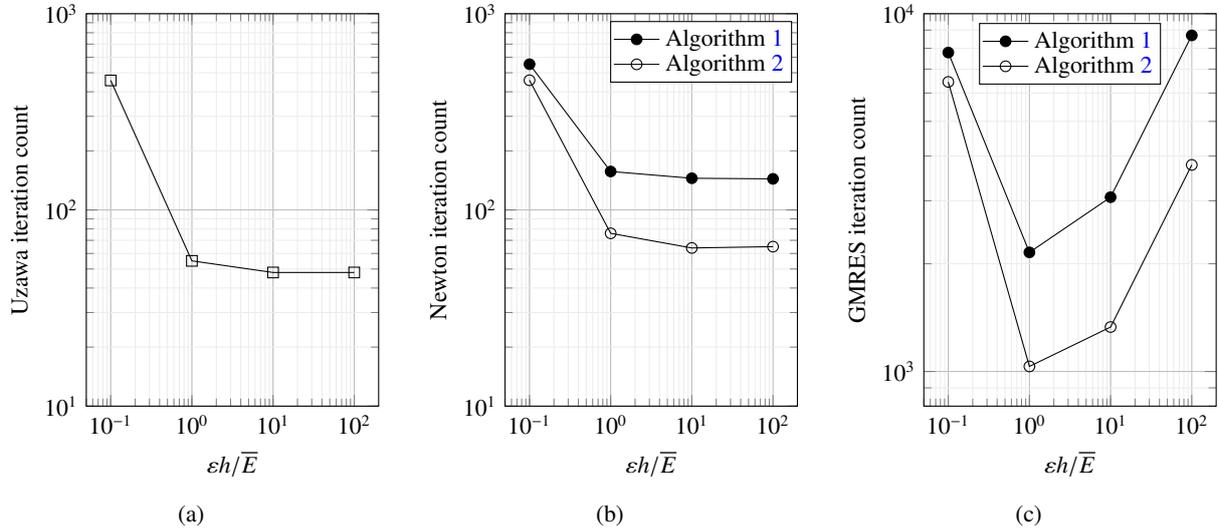
\begin{figure}
  \centering
  \small
  \begin{tikzpicture}
    \begin{axis}[
        ybar,
        bar width=20pt,
        width=\linewidth,
        height=0.4\linewidth,
        enlarge x limits=0.15,
        ymin=0, ymax=3500,
        ylabel={Average time step size [days]},
        log origin=infty,
        legend style={
            at={(0.5,-0.15)},
            anchor=north,legend columns=1},
        ylabel={{Iteration count}},
        symbolic x coords={LM,UzNSym,UzSym,Alg2},
        xtick=data,
        xticklabels={{LMM-AS},{Uz-NSym},{Uz-Sym},{Alg2}},
        xticklabels={{\shortstack{Lagrange Multiplier Method\\(Active Set)}},
            {\shortstack{\cref{alg:uzawaproc}\\(nonsymmetric)}},
            {\shortstack{\cref{alg:uzawaproc}\\(symmetric)}},
            {\shortstack{\cref{alg:alternativeAlgo}\\(symmetric)}}},
        ytick distance = 1000,
        point meta=explicit,
        nodes near coords,
        every node near coord/.append style={
            /pgf/number format/fixed,
            /pgf/number format/precision=0,
            /pgf/number format/zerofill,
            font=\small,
          },
        ybar legend,
        legend style={
            at={(0.99,0.975)},
            anchor=north east,
            legend columns=1,
            legend cell align=left,
          },
      ]
      \addplot+[
        draw=WildStrawberry!90!black,
        fill=WildStrawberry!40,
        every node near coord/.append style={text=WildStrawberry!90!black}
      ] coordinates {(LM, 133)[133] (UzNSym,  145)[145] (UzSym,  142)[142] (Alg2,  64)[64] };
      \addplot+[
        draw=WildStrawberry!80,
        fill=WildStrawberry!10,
        every node near coord/.append style={text=WildStrawberry!70}
      ] coordinates {  (LM, 3137)[3137] (UzNSym, 3070)[3070] (UzSym, 2952)[2952] (Alg2, 1331)[1331]};
      \addplot+[
        draw=gray,
        fill=lightgray!10,
        every node near coord/.append style={text=gray!85!black}
      ] coordinates { (LM,18)[18] (UzNSym, 48)[48] (UzSym, 51)[51] };
      \legend{ {Newton}, {GMRES}, {Uzawa/Active Set}}
    \end{axis}
  \end{tikzpicture}
  \caption{Stick-Slip-Open Case: Number of linear and nonlinear iterations for all examined algorithms. The number of Uzawa iterations or Active Set loops is shown in the rightmost bar for each corresponding algorithm. } \label{fig:SSO_bar}
\end{figure}
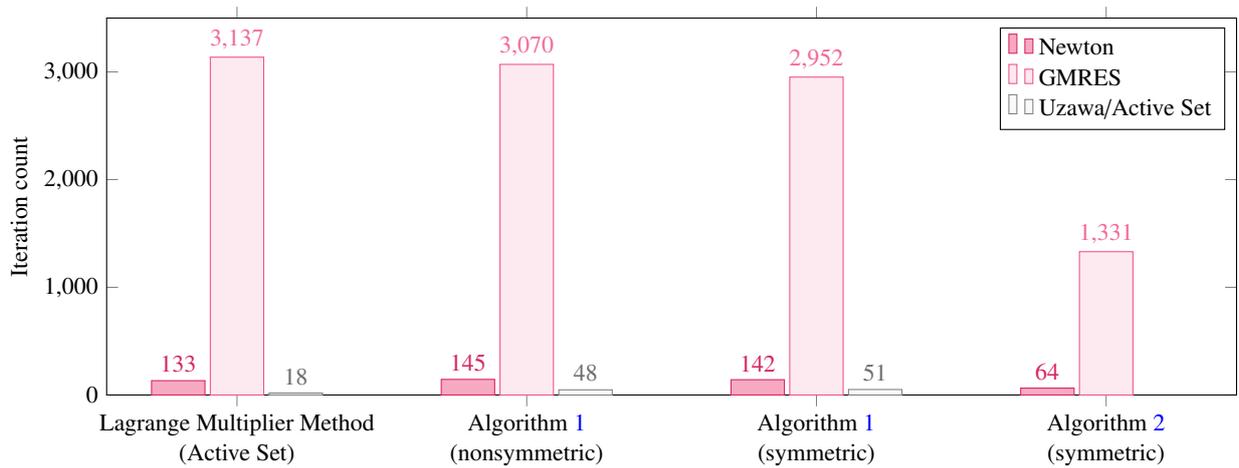
\subsubsection{Constant Slip Solution}
This analytical benchmark was initially proposed in \cite{Bor08}. The representation of the 2D model domain is shown in Figure \ref{fig:CS_dom}, where the lower boundary is constrained, the circled corner is also constrained in the vertical direction, and a uniform displacement is imposed on the upper boundary.
The simulation is carried out using a single time step, and the mesh consists of 3,610 hexahedral elements and 4,680 nodes.
The material is homogeneous, characterized by specific elastic parameters $E=250$ MPa and $\nu=0.3$. Coulomb's friction parameters are defined as $\theta=5.71$ and $c=0$, resulting in a friction coefficient of 0.1. The solution represents a constant sliding on the fracture at a defined value $\|\tensorOne{g}_T\|_2 = 0.1 \sqrt{2}$.
Because the shear vector direction is predetermined in a 2D setting, we conduct the simulation using a 3D mesh to effectively capture the nonlinearity introduced by the friction direction's dependence on the relative displacement rate.
The behavior of the considered solving strategies is notably different under the conditions analyzed in this problem. Ignoring the non-symmetric part results in a significant increase in the total number of nonlinear iterations. This is primarily due to the fact that, for each nonlinear iteration, changes in the tangential stress strongly influence the value of the normal stress to satisfy the equilibrium of the two bodies (\cref{fig:CS}).
The comparison presented in \cref{fig:CS_bar} shows that the performance of the Augmented Lagrangian Method is comparable to that of the Lagrange Multiplier with Active Set method, provided that nonsymmetry is adequately taken into account.
\begin{figure}
  \small
  \begin{subfigure}[b]{.33\linewidth}
    \centering
    \begin{tikzpicture}
      \pgfplotsset{scaled y ticks=false}
      \begin{loglogaxis}[ 
        width=\linewidth,height=1.25\linewidth,
        grid=both,
        major grid style={line width=0.4pt,draw=gray!50},
        minor grid style={line width=0.2pt,draw=gray!15},
        xmin=0.005,xmax=200.0,ymin=1,ymax=100.0,
        xlabel={{ ${\varepsilon h}/{\overline{E}}$ } },
        ylabel={Uzawa iteration count},
        ylabel near ticks,xlabel near ticks,
        legend style={
            cells={anchor=west},
            inner sep=1pt,
            outer sep=0pt,
            row sep=-2pt
          }]
        \addplot [black,mark=square] table [x=penalty,y=nUz1] {./constantSlip.txt};

      \end{loglogaxis}
    \end{tikzpicture}
    \caption{}
  \end{subfigure}
  \hfill
  \begin{subfigure}[b]{.33\linewidth}
    \centering
    \begin{tikzpicture}
      \pgfplotsset{scaled y ticks=false}
      \begin{loglogaxis}[ 
        width=\linewidth,height=1.25\linewidth,
        grid=both,
        major grid style={line width=0.4pt,draw=gray!50},
        minor grid style={line width=0.2pt,draw=gray!15},
        xmin=0.005,xmax=200.0,ymin=1,ymax=100.0,
        xlabel={{ ${\varepsilon h}/{\overline{E}}$ } },
        ylabel={Newton iteration count},
        ylabel near ticks,xlabel near ticks,
        legend style={
            cells={anchor=west},
            inner sep=1pt,
            outer sep=0pt,
            row sep=-2pt
          }]
        \addplot [black,mark=*] table [x=penalty,y=nNL1] {./constantSlip.txt};
        \addplot [black,mark=o] table [x=penalty,y=nNL2] {./constantSlip.txt};
        \legend{{\cref{alg:uzawaproc}},{\cref{alg:alternativeAlgo}}}
      \end{loglogaxis}
    \end{tikzpicture}
    \caption{}
  \end{subfigure}
  \hfill
  \begin{subfigure}[b]{.33\linewidth}
    \centering
    \begin{tikzpicture}
      \pgfplotsset{scaled y ticks=false}
      \begin{loglogaxis}[  
        width=\linewidth,height=1.25\linewidth,
        grid=both,
        major grid style={line width=0.4pt,draw=gray!50},
        minor grid style={line width=0.2pt,draw=gray!15},
        xmin=0.005,xmax=200.0,ymin=50,ymax=4000.0,
        xlabel={{ ${\varepsilon h}/{\overline{E}}$ } },
        ylabel={GMRES iteration count},
        ylabel near ticks,xlabel near ticks,
        legend style={
            at={(0.5, 0.975)},
            anchor=north,
            cells={anchor=west},
            inner sep=1pt,
            outer sep=0pt,
            row sep=-2pt
          }]
        \addplot [black,mark=*] table [x=penalty,y=nL1] {./constantSlip.txt};
        \addplot [black,mark=o] table [x=penalty,y=nL2] {./constantSlip.txt};
        \legend{{\cref{alg:uzawaproc}},{\cref{alg:alternativeAlgo}}}
      \end{loglogaxis}
    \end{tikzpicture}
    \caption{}
  \end{subfigure}
  \caption{Constant Slip Solution: Number of Uzawa (a), Newton (b), and GMRES (c) iterations as a function of the penalty parameter, with its default value equal to $10 \bar{E} / h$.}\label{fig:CS}
\end{figure}
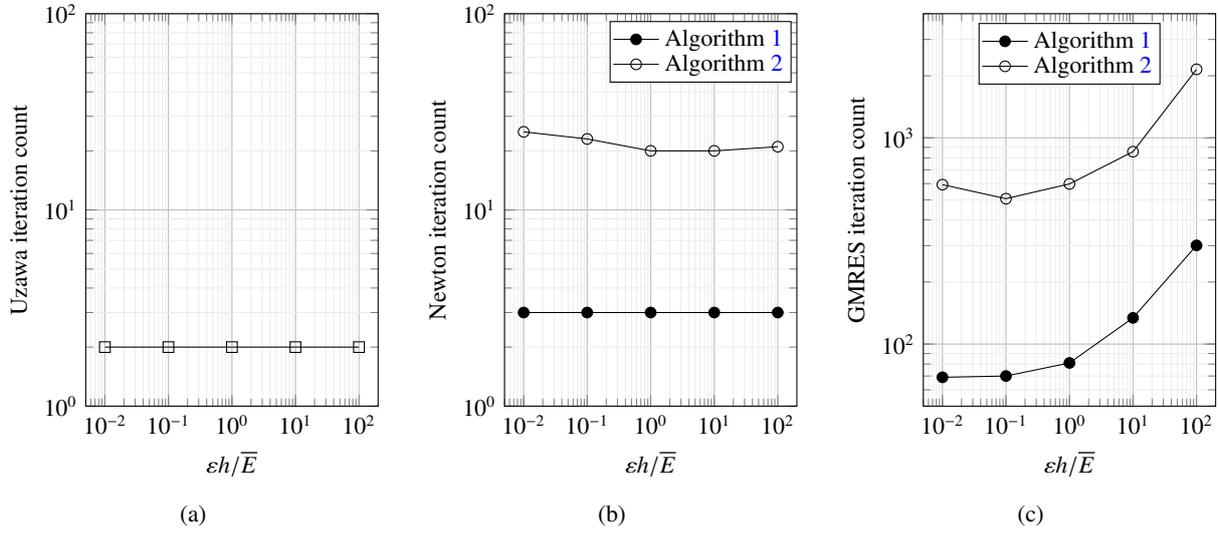
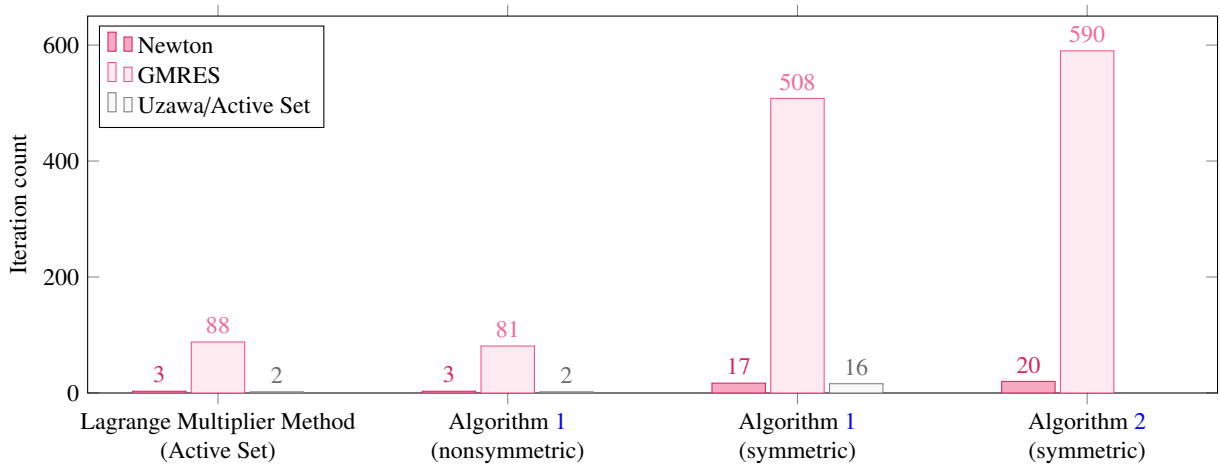
\begin{figure}
  \centering
  \small
  \begin{tikzpicture}
    \begin{axis}[
        ybar,
        bar width=20pt,
        width=\linewidth,
        height=0.4\linewidth,
        enlarge x limits=0.15,
        ymin=0, ymax=650,
        ylabel={Average time step size [days]},
        log origin=infty,
        legend style={
            at={(0.5,-0.15)},
            anchor=north,legend columns=1},
        ylabel={{Iteration count}},
        symbolic x coords={LM,UzNSym,UzSym,Alg2},
        xtick=data,
        xticklabels={{LMM-AS},{Uz-NSym},{Uz-Sym},{Alg2}},
        xticklabels={{\shortstack{Lagrange Multiplier Method\\(Active Set)}},
            {\shortstack{\cref{alg:uzawaproc}\\(nonsymmetric)}},
            {\shortstack{\cref{alg:uzawaproc}\\(symmetric)}},
            {\shortstack{\cref{alg:alternativeAlgo}\\(symmetric)}}},
        ytick distance = 200,
        point meta=explicit,
        nodes near coords,
        every node near coord/.append style={
            /pgf/number format/fixed,
            /pgf/number format/precision=0,
            /pgf/number format/zerofill,
            font=\small,
          },
        ybar legend,
        legend style={
            at={(0.01,0.975)},
            anchor=north west,
            legend columns=1,
            legend cell align=left,
          },
      ]
      \addplot+[
        draw=WildStrawberry!90!black,
        fill=WildStrawberry!40,
        every node near coord/.append style={text=WildStrawberry!90!black}
      ] coordinates {(LM, 3)[3] (UzNSym,  3)[3] (UzSym,  17)[17] (Alg2, 20)[20] };
      \addplot+[
        draw=WildStrawberry!80,
        fill=WildStrawberry!10,
        every node near coord/.append style={text=WildStrawberry!70}
      ] coordinates { (LM, 88)[88] (UzNSym, 81)[81] (UzSym, 508)[508] (Alg2, 590)[590] };
      \addplot+[
        draw=gray,
        fill=lightgray!10,
        every node near coord/.append style={text=gray!85!black}
      ] coordinates {(LM,2)[2] (UzNSym,   2)[2] (UzSym,  16)[16] };
      \legend{ {Newton}, {GMRES}, {Uzawa/Active Set}}
    \end{axis}
  \end{tikzpicture}
  \caption{Constant Slip Solution: Number of linear and nonlinear iterations for all examined algorithms. The number of Uzawa iterations or Active Set loops is shown in the rightmost bar for each corresponding algorithm.}
  \label{fig:CS_bar}
\end{figure}
\subsubsection{T-Crack}
In this example, we simulate two fractures intersecting at right angles within a 3D infinite domain, subjected to a constant uniaxial compressive remote stress. The problem setup is thoroughly described in \cite{PhaNapGraKap03}, where it is addressed using the symmetric Galerkin boundary element method.
Specifically, we consider two intersecting fractures under a remote compressive stress constraint, as illustrated in Figure \ref{fig:TF_dom}. Both fractures lie within an infinite, homogeneous, isotropic, and elastic medium.
The mesh contains $300 \times 300 \times 2$ hexahedral elements in the $x$, $y$, and $z$ directions, respectively. Near the cracks, the grid is refined with an element size of 10 m, while in the rest of the domain the element size is 45 m.
The simulation, performed over ten loading steps, models a pressurized vertical fracture intersecting a horizontal fracture at its midpoint.
The combination of uniaxial compression, frictional contact, and the opening of the vertical fracture induces mechanical deformations in the surrounding rock, leading to the sliding of the horizontal fracture.
In this test case, we can observe the effectiveness of this discretization in managing fracture intersections. The convergence of Algorithms \ref{alg:uzawaproc} and \ref{alg:alternativeAlgo} is quite similar.
Figure \ref{fig:TF} demonstrates that the optimal value of $\varepsilon$ results from a trade-off aimed at maximizing the penalty value to reduce nonlinear iterations without compromising the conditioning of the system $\hat{A}_{uu}$.
\begin{figure}
  \small
  \begin{subfigure}[b]{.33\linewidth}
    \centering
    \begin{tikzpicture}
      \pgfplotsset{scaled y ticks=false}
      \begin{loglogaxis}[ 
        width=\linewidth,height=1.25\linewidth,
        grid=both,
        major grid style={line width=0.4pt,draw=gray!50},
        minor grid style={line width=0.2pt,draw=gray!15},
        xmin=0.05,xmax=200.0,ymin=4,ymax=100.0,
        xlabel={{ ${\varepsilon h}/{\overline{E}}$ } },
        ylabel={Uzawa iteration count},
        ylabel near ticks,xlabel near ticks,
        legend style={
            cells={anchor=west},
            inner sep=1pt,
            outer sep=0pt,
            row sep=-2pt
          }]
        \addplot [black,mark=square] table [x=penalty,y=nUz1] {./tFrac.txt};
      \end{loglogaxis}
    \end{tikzpicture}
    \caption{}
  \end{subfigure}
  \hfill
  \begin{subfigure}[b]{.33\linewidth}
    \centering
    \begin{tikzpicture}
      \pgfplotsset{scaled y ticks=false}
      \begin{loglogaxis}[ 
        width=\linewidth,height=1.25\linewidth,
        grid=both,
        major grid style={line width=0.4pt,draw=gray!50},
        minor grid style={line width=0.2pt,draw=gray!15},
        xmin=0.05,xmax=200.0,ymin=4,ymax=100.0,
        xlabel={{ ${\varepsilon h}/{\overline{E}}$ } },
        ylabel={Newton iteration count},
        ylabel near ticks,xlabel near ticks,
        legend style={
            cells={anchor=west},
            inner sep=1pt,
            outer sep=0pt,
            row sep=-2pt
          }]
        \addplot [black,mark=square] table [x=penalty,y=nUz1] {./tFrac.txt};
        \addplot [black,mark=*] table [x=penalty,y=nNL1] {./tFrac.txt};
        \addplot [black,mark=o] table [x=penalty,y=nNL2] {./tFrac.txt};
        \legend{{\cref{alg:uzawaproc}},{\cref{alg:alternativeAlgo}}}
      \end{loglogaxis}
    \end{tikzpicture}
    \caption{}
  \end{subfigure}
  \hfill
  \begin{subfigure}[b]{.33\linewidth}
    \centering
    \begin{tikzpicture}
      \pgfplotsset{scaled y ticks=false}
      \begin{loglogaxis}[   
        width=\linewidth,height=1.25\linewidth,
        grid=both,
        major grid style={line width=0.4pt,draw=gray!50},
        minor grid style={line width=0.2pt,draw=gray!15},
        xmin=0.05,xmax=200.0,ymin=90,ymax=2000.0,
        xlabel={{ ${\varepsilon h}/{\overline{E}}$ } },
        ylabel={GMRES iteration count},
        ylabel near ticks,xlabel near ticks,
        legend style={
            cells={anchor=west},
            inner sep=1pt,
            outer sep=0pt,
            row sep=-2pt
          }]
        \addplot [black,mark=*] table [x=penalty,y=nL1] {./tFrac.txt};
        \addplot [black,mark=o] table [x=penalty,y=nL2] {./tFrac.txt};
        \legend{{\cref{alg:uzawaproc}},{\cref{alg:alternativeAlgo}}}
      \end{loglogaxis}
    \end{tikzpicture}
    \caption{}
  \end{subfigure}
  \caption{T-Crack Case: Number of Uzawa (a), Newton (b), and GMRES (c) iterations as a function of the penalty parameter.}\label{fig:TF}
\end{figure}
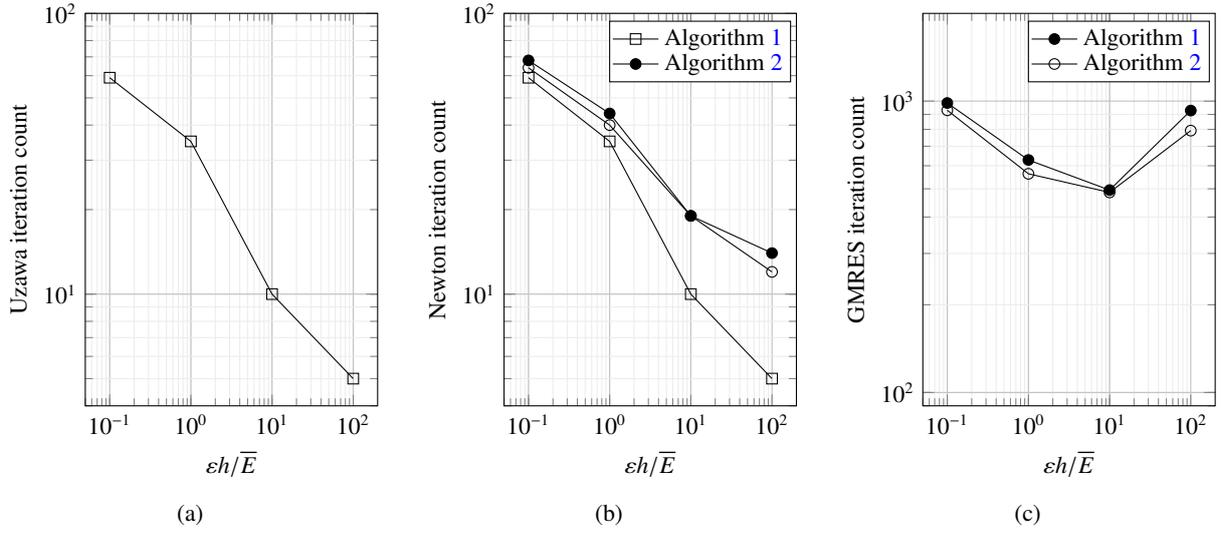
\begin{figure}
  \centering
  \small
  \begin{tikzpicture}
    \begin{axis}[
        ybar,
        bar width=20pt,
        width=\linewidth,
        height=0.4\linewidth,
        enlarge x limits=0.15,
        ymin=0, ymax=1350,
        ylabel={Average time step size [days]},
        log origin=infty,
        legend style={
            at={(0.5,-0.15)},
            anchor=north,legend columns=1},
        ylabel={{Iteration count}},
        symbolic x coords={LM,UzNSym,UzSym,Alg2},
        xtick=data,
        xticklabels={{LMM-AS},{Uz-NSym},{Uz-Sym},{Alg2}},
        xticklabels={{\shortstack{Lagrange Multiplier Method\\(Active Set)}},
            {\shortstack{\cref{alg:uzawaproc}\\(nonsymmetric)}},
            {\shortstack{\cref{alg:uzawaproc}\\(symmetric)}},
            {\shortstack{\cref{alg:alternativeAlgo}\\(symmetric)}}},
        ytick distance = 400,
        point meta=explicit,
        nodes near coords,
        every node near coord/.append style={
            /pgf/number format/fixed,
            /pgf/number format/precision=0,
            /pgf/number format/zerofill,
            font=\small,
          },
        ybar legend,
        legend style={
            at={(0.99,0.975)},
            anchor=north east,
            legend columns=1,
            legend cell align=left,
          },
      ]
      \addplot+[
        draw=WildStrawberry!90!black,
        fill=WildStrawberry!40,
        every node near coord/.append style={text=WildStrawberry!90!black}
      ] coordinates {(LM, 51)[51] (UzNSym,  19)[19] (UzSym,  31)[31] (Alg2,  19)[19] };
      \addplot+[
        draw=WildStrawberry!80,
        fill=WildStrawberry!10,
        every node near coord/.append style={text=WildStrawberry!70}
      ] coordinates {  (LM, 1190)[1190] (UzNSym, 495)[495] (UzSym, 811)[811] (Alg2, 486)[486]};
      \addplot+[
        draw=gray,
        fill=lightgray!10,
        every node near coord/.append style={text=gray!85!black}
      ] coordinates { (LM,11)[11] (UzNSym, 10)[10] (UzSym, 24)[24] };
      \legend{ {Newton}, {GMRES}, {Uzawa/Active Set}}
    \end{axis}
  \end{tikzpicture}
  \caption{T-Crack Case: Number of linear and nonlinear iterations for all examined algorithms. The number of Uzawa iterations or Active Set loops is shown in the rightmost bar for each corresponding algorithm.}
  \label{fig:TF_bar}
\end{figure}
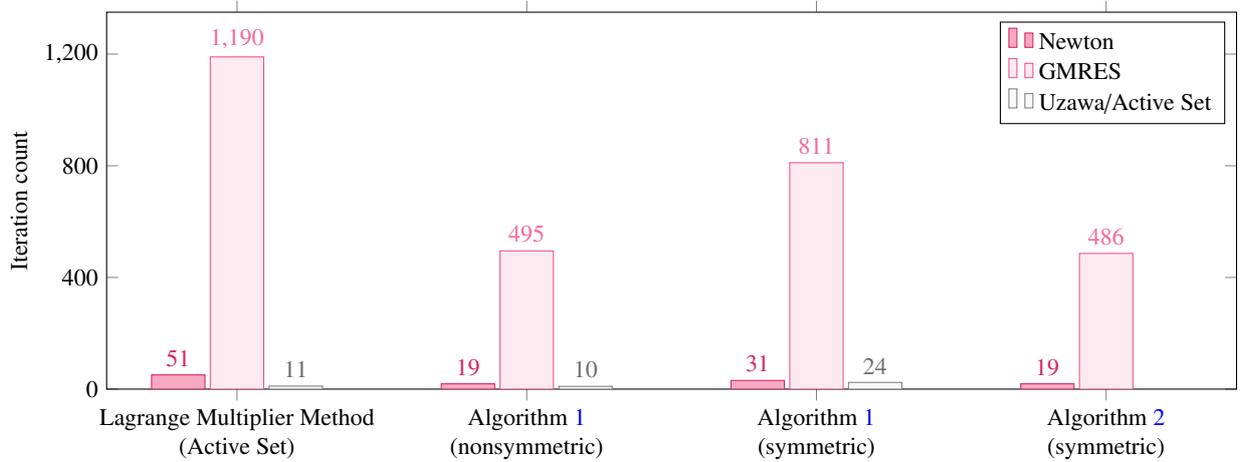
From Figure \ref{fig:TF_bar}, it can be observed that while the \cref{alg:uzawaproc,alg:alternativeAlgo}) are essentially equivalent, the Lagrange Multiplier -AS method struggles to converge when slip-stick conditions occur across multiple elements after the first time step of the simulation.
The Lagrange Multiplier-Active Set method is indeed very effective for managing stick or stick-open conditions, but it exhibits poor convergence behavior when numerous elements transition to a slip state.
In particular, the active set procedure is not very effective in determining the set of elements that share the same state.
\subsection{Field Scale Problem}
The results section concludes with the examination of the field-scale problem illustrated in \cref{fig:FSP_dom}. Here, we consider a highly complex fault system featuring nine faults discretized using a conforming mesh. These faults exhibit different orientations, are curved, and have multiple intersections (\cref{fig:FSP_faults_dom}). The mesh is obtained by extruding in the $y$-direction an unstructured triangulation of the vertical cross-section ($x$-$z$ plane) recently used in \cite{Bos_etal21}. The 3D mesh consists of 622,950 wedge elements, 312,885 nodes, and 12,450 interface elements.
The domain is subdivided into five distinct geological layers, as highlighted in the figure with different colors. The initial stress state is determined by balancing the gravitational forces, and the mechanical properties adhere to the relationships outlined in \cite{BauFerGamTea02}.We simulate pressurization in the red area to study a synthetic scenario of the geological carbon storage process. The pressure history is linear and reaches its maximum value of 9 MPa by the end of the simulation.
The problem has been effectively solved using the Augmented Lagrangian Method, achieving convergence in 3 Uzawa iterations, with a total of 13 Newton iterations. It is worth mentioning that the Lagrange Multiplier Method with Active Set does not converge for this specific case, highlighting the superior robustness of the ALM compared to the other method.
The solution to this problem is presented in \cref{fig:FSP_dz,fig:FSP_tn,fig:FSP_tt,fig:FSP_fs}. In particular, \cref{fig:FSP_dz} illustrates the uplift displacement resulting from the surface subsiding after the injection of CO2. The other figures aim to display the solution along the faults.
Firstly, it can be observed that the pressurization of the reservoir activates the faults. Specifically, there are surface elements that slide and open in the interface area adjacent to the reservoir, as shown in \cref{fig:FSP_fs}. \cref{fig:FSP_tn,fig:FSP_tt} represent the values of normal and tangential traction, respectively. These figures demonstrate that the face bubble stabilization effectively stabilizes the problem, producing an oscillation-free solution even in this complex geological configuration with high property heterogeneity.
For the sake of completeness, we also include the slip solution on the faults (\cref{fig:FSP_slip}), from which it can be observed that there are slips on the order of centimeters where the faults are activated.
\begin{figure}
  \centering
  \begin{subfigure}[b]{.49\linewidth}
    \centering
    \includegraphics[height=10em]{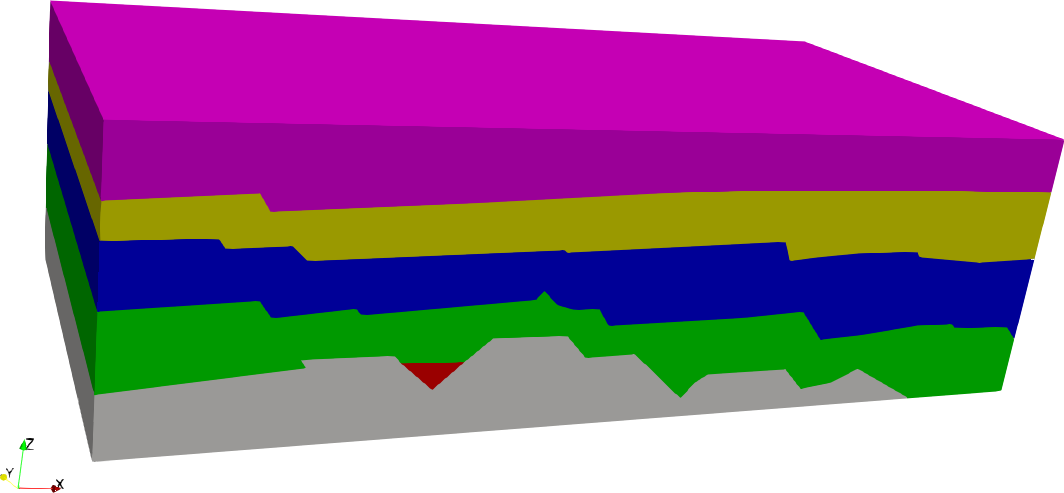}
    \caption{}
  \end{subfigure}
  \hfill
  \begin{subfigure}[b]{.45\linewidth}
    \small
    \centering
    \begingroup
    \renewcommand{\arraystretch}{1.4} 
    \begin{tabular}{lll}
      \toprule
      Quantity                  & Value                 & Unit \\
      \midrule
      Young's modulus ($E$)     & \cite{BauFerGamTea02} & [Pa] \\
      Poisson's ratio ($\nu$)   & $0.3$                 & [-]  \\
      Friction angle ($\theta$) & $30.0$                & [-]  \\
      \bottomrule
    \end{tabular}
    \endgroup
    \caption{}
  \end{subfigure}
  \caption{Field Scale Problem: (a) 3D computational domain and (b) physical properties. The 3D region is divided into six distinct layers indicated by the colors.}
  \label{fig:FSP_dom}
\end{figure}
\begin{figure}
  \centering
  \begin{subfigure}[b]{.49\linewidth}
    \centering
    \includegraphics[width=\linewidth]{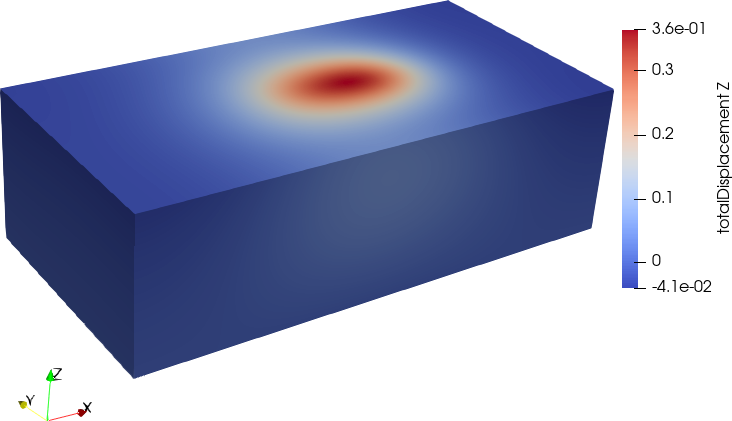}
    \caption{}
    \label{fig:FSP_dz}
  \end{subfigure}
  \hfill
  \begin{subfigure}[b]{.49\linewidth}
    \centering
    \includegraphics[width=\linewidth]{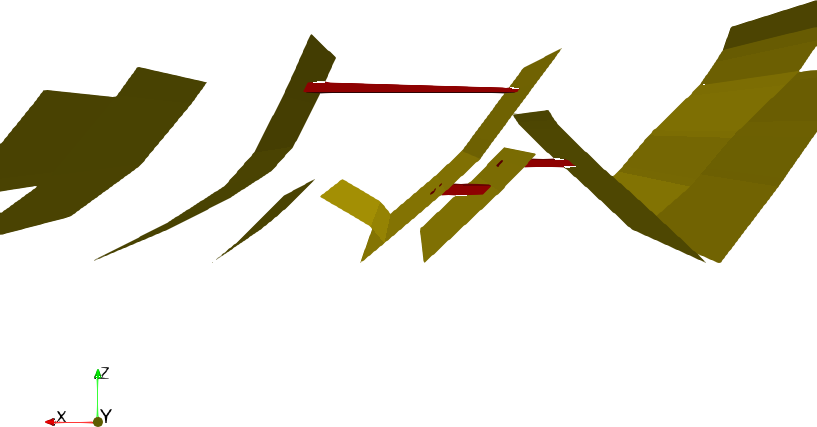}
    \caption{}
    \label{fig:FSP_faults_dom}
  \end{subfigure}
  \caption{Field Scale Problem: (a) Vertical displacement field; (b) 2D fault domain and reservoir regions.}
\end{figure}
\begin{figure}
  \centering
  \begin{subfigure}[b]{.49\linewidth}
    \centering
    \includegraphics[width=\linewidth]{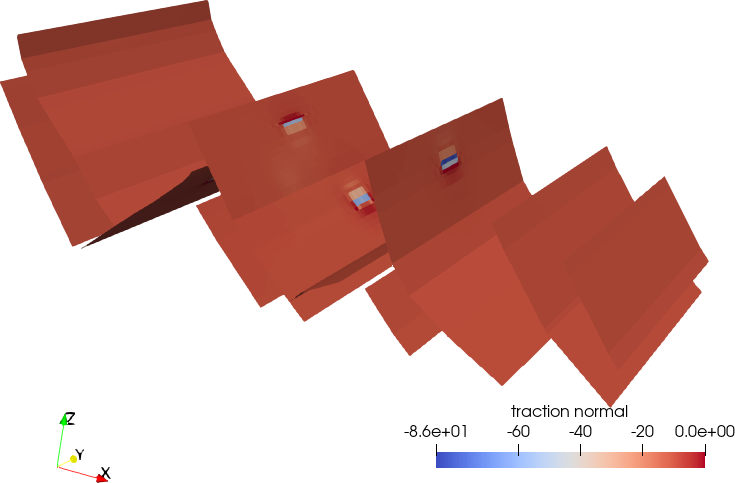}
    \caption{}
    \label{fig:FSP_tn}
  \end{subfigure}
  \hfill
  \begin{subfigure}[b]{.49\linewidth}
    \centering
    \includegraphics[width=\linewidth]{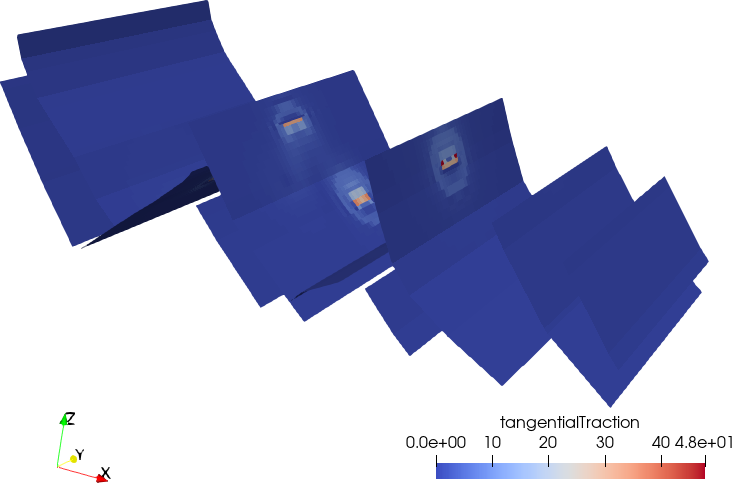}
    \caption{}
    \label{fig:FSP_tt}
  \end{subfigure}
  \caption{Field Scale Problem: Contour for (a) normal traction and (b) tangential traction.}
\end{figure}
\begin{figure}
  \centering
  \hfill
  \begin{subfigure}[b]{.49\linewidth}
    \centering
    \includegraphics[width=\linewidth]{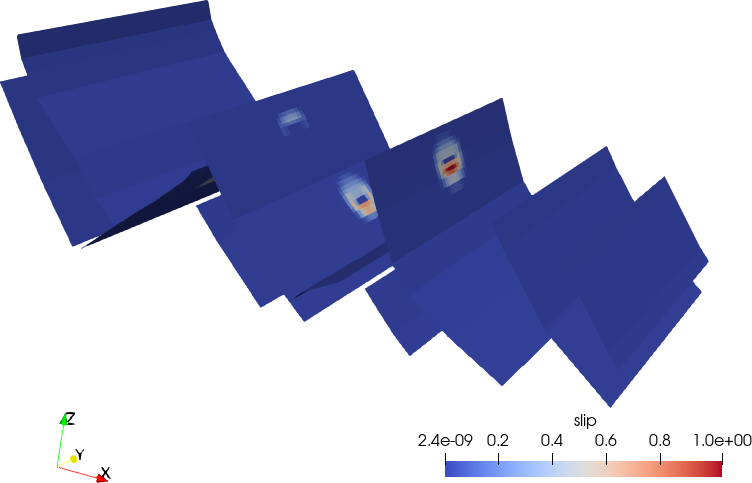}
    \caption{}
    \label{fig:FSP_slip}
  \end{subfigure}
  \begin{subfigure}[b]{.49\linewidth}
    \centering
    \includegraphics[width=\linewidth]{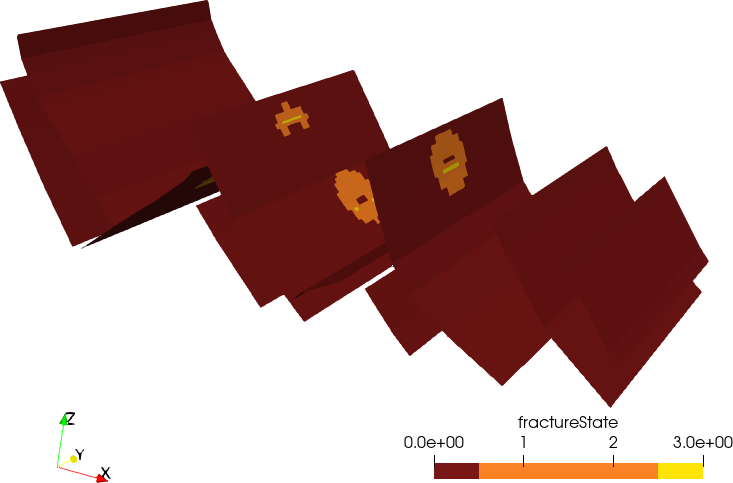}
    \caption{}
    \label{fig:FSP_fs}
  \end{subfigure}
  \caption{Field Scale Problem: Contour for (a) tangential slip and (b) fracture state.}
\end{figure}
%
%
\section{Conclusions}
In this work, we have introduced a stabilized Augmented Lagrangian formulation for contact mechanics, discretized using low-order linear piecewise displacement and constant cell-centered tractions.
The stabilization is achieved by enriching the displacement space with face bubble functions. This approach has proven effective in satisfying the inf-sup condition for all element types (hexahedra, wedges, and tetrahedra) and can be naturally incorporated into the iterative Augmented Lagrangian procedure.
Numerical results demonstrate that the proposed method  is both accurate and robust when dealing with the strong non linearity arising from contact mechanics applied to faults in subsurface environments when addressing the strong nonlinearities associated with contact mechanics on faults in subsurface environments. The nonlinear solver exhibits enhanced convergence behavior, showing not only greater robustness but, in some cases, faster convergence compared to methods previously reported in the literature.

\section{Data availability}
The method presented in this work was implemented in the GEOS open source simulation framework\citep{Geos24} and results were  produced using GEOS commit \href{https://github.com/GEOS-DEV/GEOS/commit/555a6cbbf8d2ae686cc22bef709a5e904030936d}{\texttt{555a6cbbf8d2ae686cc22bef709a5e904030936d}}.

\section*{Acknowledgments}
This work was funded by TotalEnergies and Chevron through the FC-MAELSTROM Project and by the Office of Fossil Energy and Carbon Management of the United States Department of Energy through project FEW0287. A portion of this work was performed under the auspices of the U.S. Department of Energy by Lawrence Livermore National Laboratory under Contract DE-AC52-07NA27344.
We thank Michael A. Puso and Jerome M. Solberg for their insights and the fruitful discussions.

\appendix

\section{Finite element matrices}\label{sec:appendixA}
Let $\tensorOne{u}^{h,k+1}_n = ( \sum_i u_i \tensorOne{\eta}_i + \sum_j \hat{u}_j \tensorOne{\beta}_j) $ and $\tensorOne{t}^{h,k}_n = \sum_k t_i \tensorOne{\mu}_i$.
%
%
Given the following two projection operators:
\begin{align}
   & P_{\tensorOne{n}_f} : \mathbb{R}^3 \rightarrow \mathbb{R}^3,
  \tensorOne{v} \mapsto P_{\tensorOne{n}_f} \tensorOne{v} = (\tensorOne{n}_f \otimes \tensorOne{n}_f) \tensorOne{v}, \\
   & P_{\tensorOne{n}_f}^{\perp} : \mathbb{R}^3 \rightarrow \mathbb{R}^3,
  \tensorOne{v} \mapsto P_{\tensorOne{n}_f}^{\perp} \tensorOne{v} = (\tensorTwo{1} - \tensorOne{n}_f \otimes \tensorOne{n}_f) \tensorOne{v},
\end{align}
the following relationships hold true:
\begin{align}
   & \llbracket \tensorOne{u}^{h,k+1}_n \rrbracket
  = P_{\tensorOne{n}_f} \llbracket \tensorOne{u}^{h,k+1}_n \rrbracket
  + P_{\tensorOne{n}_f}^{\perp} \llbracket \tensorOne{u}^{h,k+1}_n \rrbracket
  = \underbrace{
    \sum_i u_i P_{\tensorOne{n}_f} \llbracket \tensorOne{\eta}_i \rrbracket
    + \sum_j \hat{u}_j P_{\tensorOne{n}_f} \llbracket \tensorOne{\beta}_j \rrbracket
  }_{:= \tensorOne{g}_N(\tensorOne{u}^h)}
  + \underbrace{
    \sum_i u_i P_{\tensorOne{n}_f}^{\perp}  \llbracket \tensorOne{\eta}_i \rrbracket
    + \sum_j \hat{u}_j P_{\tensorOne{n}_f}^{\perp}  \llbracket \tensorOne{\beta}_j \rrbracket
  }_{:= \tensorOne{g}_T(\tensorOne{u}^h)},
  \\
   & g_N(\tensorOne{u}^{h,k+1}_n)
  = \tensorOne{g}_N(\tensorOne{u}^{h,k+1}_n) \cdot \tensorOne{n}_f
  = \sum_i u_i  \llbracket \tensorOne{\eta}_i \rrbracket \cdot \tensorOne{n}_f
  + \sum_j \hat{u}_j  \llbracket \tensorOne{\beta}_j \rrbracket \cdot \tensorOne{n}_f,
  \\
   & \Delta_n \tensorOne{g}_T (\tensorOne{u}^{h,k+1}_n)
  = \sum_i \Delta_n u_i P_{\tensorOne{n}_f}^{\perp}  \llbracket \tensorOne{\eta}_i \rrbracket
  + \sum_j \Delta_n \hat{u}_j P_{\tensorOne{n}_f}^{\perp}  \llbracket \tensorOne{\beta}_j \rrbracket,
  \\
   & \hat{\tensorOne{t}} ( g_N ,\Delta_n \tensorOne{g}_T)
  = \hat{\tensorOne{t}}_N ( g_N )
  + \hat{\tensorOne{t}}_T ( g_N ,\Delta_n \tensorOne{g}_T).
\end{align}
%
Then, the global residual block vectors associated with \cref{eq:NR1,eq:NR1bubbles} are expressed as follows:
\begin{linenomath}
  \begin{subequations}
    \begin{align}
      [\Vec{r}_b]_i
       &
      = ( \nabla^s \tensorOne{\beta}_i, \tensorTwo{\sigma}(\tensorOne{u}^{h,k+1}_n) )_{\Omega}
      + ( \llbracket \tensorOne{\beta}_i \rrbracket, \hat{\tensorOne{t}}_N ( g_N(\tensorOne{u}^{h,k+1}_n) )
      + \hat{\tensorOne{t}}_T ( g_N(\tensorOne{u}^{h,k+1}_n), \Delta_n \tensorTwo{g}_T(\tensorOne{u}^{h,k+1}_n) )_{\Gamma_f}
       & \forall i \in \{1, \ldots, n_b \} ,
      \\
      [\Vec{r}_u]_i
       &
      = ( \nabla^s \tensorOne{\eta}_i, \tensorTwo{\sigma}(\tensorOne{u}^{h,k+1}_n) )_{\Omega}
      + ( \llbracket \tensorOne{\eta}_i \rrbracket, \hat{\tensorOne{t}}_N ( g_N(\tensorOne{u}^{h,k+1}_n) )
      + \hat{\tensorOne{t}}_T ( g_N(\tensorOne{u}^{h,k+1}_n), \Delta_n \tensorTwo{g}_T(\tensorOne{u}^{h,k+1}_n) )_{\Gamma_f}
      - ( \tensorOne{\eta}_i, \bar{\tensorOne{t}} )_{\Gamma_{\sigma}}
       & \forall i \in \{1, \ldots, n_u \} ,
    \end{align}
  \end{subequations}
\end{linenomath}
with $n_b = \text{dim}( \boldsymbol{\mathcal{U}}^b )$ and $n_u = \text{dim}( \boldsymbol{\mathcal{U}}^{h,1} )$.
%
and, the global expressions for the submatrices appearing in the Jacobian matrix are as follows:
\begin{align}
  [\Mat{A}_{bb}]_{ij}
  &
  = \left( \nabla^s \tensorOne{\beta}_i, \evaluate{\frac{\partial \tensorTwo{\sigma}}{\partial \tensorTwo{\epsilon}}}{}{\ell}                                                                                                                                                                                                                                                          \nabla^s \tensorOne{\beta}_j \right)_{\Omega}
  + \left( \llbracket \tensorOne{\beta}_i \rrbracket, \evaluate{\frac{\partial \hat{\tensorOne{t}}_N}{\partial \tensorOne{g}_N}}{}{\ell}                                                                                                                                                                                                                                                           P_{\tensorOne{n}_f} \llbracket \tensorOne{\beta}_j \rrbracket \right)_{\Gamma_f}
  \nonumber \\
  &
  + \left( \llbracket \tensorOne{\beta}_i \rrbracket, \evaluate{ \frac{\partial \hat{\tensorOne{t}}_T}{\partial \tensorOne{g}_N}}{}{\ell}                                                                                                                                                                                                                                                          P_{\tensorOne{n}_f} \llbracket \tensorOne{\beta}_j \rrbracket \right)_{\Gamma_f}
  + \left( \llbracket \tensorOne{\beta}_i \rrbracket, \evaluate{\frac{\partial \hat{\tensorOne{t}}_T}{\partial \Delta_n \tensorOne{g}_T }}{}{\ell}                                                                                                                                                                                                                                                          P_{\tensorOne{n}_f}^{\perp} \llbracket \tensorOne{\beta}_j \rrbracket \right)_{\Gamma_f}
  & \forall (i,j) \in \{1, \ldots, n_b \} \times \{1, \ldots, n_b \} ,
  \\
  [\Mat{A}_{bu}]_{ij}
  &
  = \left( \nabla^s \tensorOne{\beta}_i, \evaluate{\frac{\partial \tensorTwo{\sigma}}{\partial \tensorTwo{\epsilon}}}{}{\ell}                                                                                                                                                                                                                                                          \nabla^s \tensorOne{\eta}_j \right)_{\Omega}
  + \left( \llbracket \tensorOne{\beta}_i \rrbracket, \evaluate{\frac{\partial \hat{\tensorOne{t}}_N}{\partial \tensorOne{g}_N}}{}{\ell}                                                                                                                                                                                                                                                           P_{\tensorOne{n}_f} \llbracket \tensorOne{\eta}_j \rrbracket \right)_{\Gamma_f}
  \nonumber \\
  &
  + \left( \llbracket \tensorOne{\beta}_i \rrbracket, \evaluate{ \frac{\partial \hat{\tensorOne{t}}_T}{\partial \tensorOne{g}_N}}{}{\ell}                                                                                                                                                                                                                                                          P_{\tensorOne{n}_f} \llbracket \tensorOne{\eta}_j \rrbracket \right)_{\Gamma_f}
  + \left( \llbracket \tensorOne{\beta}_i \rrbracket, \evaluate{\frac{\partial \hat{\tensorOne{t}}_T}{\partial \Delta_n \tensorOne{g}_T }}{}{\ell}                                                                                                                                                                                                                                                          P_{\tensorOne{n}_f}^{\perp} \llbracket \tensorOne{\eta}_j \rrbracket \right)_{\Gamma_f}
  & \forall (i,j) \in \{1, \ldots, n_b \} \times \{1, \ldots, n_u \} ,
  \\
  [\Mat{A}_{ub}]_{ij}
  &
  = \left( \nabla^s \tensorOne{\eta}_i, \evaluate{\frac{\partial \tensorTwo{\sigma}}{\partial \tensorTwo{\epsilon}}}{}{\ell}                                                                                                                                                                                                                                                          \nabla^s \tensorOne{\beta}_j \right)_{\Omega}
  + \left( \llbracket \tensorOne{\eta}_i \rrbracket, \evaluate{\frac{\partial \hat{\tensorOne{t}}_N}{\partial \tensorOne{g}_N}}{}{\ell}                                                                                                                                                                                                                                                           P_{\tensorOne{n}_f} \llbracket \tensorOne{\beta}_j \rrbracket \right)_{\Gamma_f}
  \nonumber \\
  &
  + \left( \llbracket \tensorOne{\eta}_i \rrbracket, \evaluate{ \frac{\partial \hat{\tensorOne{t}}_T}{\partial \tensorOne{g}_N}}{}{\ell}                                                                                                                                                                                                                                                          P_{\tensorOne{n}_f} \llbracket \tensorOne{\beta}_j \rrbracket \right)_{\Gamma_f}
  + \left( \llbracket \tensorOne{\eta}_i \rrbracket, \evaluate{\frac{\partial \hat{\tensorOne{t}}_T}{\partial \Delta_n \tensorOne{g}_T }}{}{\ell}                                                                                                                                                                                                                                                          P_{\tensorOne{n}_f}^{\perp} \llbracket \tensorOne{\beta}_j \rrbracket \right)_{\Gamma_f}
  & \forall (i,j) \in \{1, \ldots, n_u \} \times \{1, \ldots, n_b \} ,
  \\
  [\Mat{A}_{uu}]_{ij}
  &
  = \left( \nabla^s \tensorOne{\eta}_i, \evaluate{\frac{\partial \tensorTwo{\sigma}}{\partial \tensorTwo{\epsilon}}}{}{\ell}                                                                                                                                                                                                                                                          \nabla^s \tensorOne{\eta}_j \right)_{\Omega}
  + \left( \llbracket \tensorOne{\eta}_i \rrbracket, \evaluate{\frac{\partial \hat{\tensorOne{t}}_N}{\partial \tensorOne{g}_N}}{}{\ell}                                                                                                                                                                                                                                                           P_{\tensorOne{n}_f} \llbracket \tensorOne{\eta}_j \rrbracket \right)_{\Gamma_f}
  \nonumber \\
  &
  + \left( \llbracket \tensorOne{\eta}_i \rrbracket, \evaluate{ \frac{\partial \hat{\tensorOne{t}}_T}{\partial \tensorOne{g}_N}}{}{\ell}                                                                                                                                                                                                                                                          P_{\tensorOne{n}_f} \llbracket \tensorOne{\eta}_j \rrbracket \right)_{\Gamma_f}
  + \left( \llbracket \tensorOne{\eta}_i \rrbracket, \evaluate{\frac{\partial \hat{\tensorOne{t}}_T}{\partial \Delta_n \tensorOne{g}_T }}{}{\ell}                                                                                                                                                                                                                                                          P_{\tensorOne{n}_f}^{\perp} \llbracket \tensorOne{\eta}_j \rrbracket \right)_{\Gamma_f}
  & \forall (i,j) \in \{1, \ldots, n_u \} \times \{1, \ldots, n_u \} ,
\end{align}
where the partial derivatives are expanded as:
\begin{subequations}
  \begin{align}
    \evaluate{\frac{\partial \tensorTwo{\sigma}}{\partial \tensorTwo{\epsilon}}}{}{\ell}                                                                                                                                                                                                                                                         
                                                                                                                                                                                                                                                                                                                                            & =
    \mathbb{C},                                                                                                                                                                                                                                                                                                                                 \\
    \evaluate{\frac{\partial \hat{\tensorOne{t}}_N}{\partial \tensorOne{g}_N}}{}{\ell}                                                                                                                                                                                                                                                          & =
    \begin{dcases}
      \quad \varepsilon_N \tensorTwo{1} \\
      \quad \tensorTwo{0}
    \end{dcases}                                                                                                                                                                                                                                                                                                    &   &
    \begin{aligned}
       & \text{if } t^{h,k}_N + \varepsilon_N g^{h,k+1,\ell}_N \le 0, \\
       & \text{if } t^{h,k}_N + \varepsilon_N g^{h,k+1,\ell}_N > 0,
    \end{aligned}
    \\[1ex]
    \evaluate{\frac{\hat{\tensorOne{t}}_T}{\partial \tensorOne{g}_N}}{}{\ell}                                                                                                                                                                                                                                                          & =
    \begin{cases}
      \quad \tensorOne{0} \\
      \quad
      - \varepsilon_N \tan(\theta)
      \dfrac{ \tensorOne{t}^*_T }{ \| \tensorOne{t}^*_T \|_2 }
      \otimes \tensorOne{n}_f
    \end{cases}                                                                                                                                                                                                                                              &   &
    \begin{aligned}
       & \text{if } \| \tensorOne{t}^{*}_T \|_2 < \tau_{\max},   \\
       & \text{if } \| \tensorOne{t}^{*}_T \|_2 \ge \tau_{\max},
    \end{aligned}
    \\[1ex]
    \evaluate{\frac{\partial \hat{\tensorOne{t}}_T}{\partial \Delta_n \tensorOne{g}_T }}{}{\ell}                                                                                                                                                                                                                                                         
                                                                                                                                                                                                                                                                                                                                            & =
    \begin{cases}
      \quad \varepsilon_T \tensorTwo{1} \\
      \quad \varepsilon_T \tau_{\max}
      \left(
      \dfrac{ \| \tensorOne{t}^{*}_T \|_2^2 \tensorTwo{1} - \tensorOne{t}^{*}_T \otimes \tensorOne{t}^{*}_T }
      { \| \tensorOne{t}^{*}_T \|_2^3 }
      \right)
    \end{cases} &   &
    \begin{aligned}
       & \text{if } \| \tensorOne{t}^{*}_T \|_2 < \tau_{\max},   \\
       & \text{if } \| \tensorOne{t}^{*}_T \|_2 \ge \tau_{\max}.
    \end{aligned}
  \end{align}
\end{subequations}
with $\tensorOne{t}^*_T = \tensorOne{t}^{h,k}_T + \varepsilon_T \Delta_n \tensorOne{g}^{h,k+1,\ell}_T$.

\section{}\label{sec:appendixB}
In this Appendix, we gather some elementary results that are utilized to derive the theoretical findings. In particular, let us recall the inverse inequality and the trace inequality:
\begin{linenomath}
  \begin{subequations}
    \begin{align}
       & \snormH{\tensorOne{u}_h}^2 \le c \diam{\Omega}^{-2} \normL{\tensorOne{u}_h}^2,                                                                                       &  & \forall \tensorOne{u}_h \in \boldsymbol{\mathcal{U}}^{h}, \label{eq:invIneq} \\
       & \normLg{\gamma(\tensorOne{u})} \le c \left ( \diam{\Omega}^{-\frac{1}{2}} \normL{\tensorOne{u}} + \diam{\Omega}^{\frac{1}{2}} \normL{ \nabla \tensorOne{u}} \right ) &  & \forall \tensorOne{u} \in \boldsymbol{H}^1(\Omega). \label{eq:traceIneq}
    \end{align}
  \end{subequations}
\end{linenomath}
Moreover, assuming the affine mapping $F$, namely $\tensorOne{v}:= F^{-1} \circ \hat{\tensorOne{v}}$, we have the following results (\cite{BofBreFor13},  pp. 59):
\begin{linenomath}
  \begin{subequations}
    \begin{align}
       & \normL{\tensorOne{v}}^2 \le c \meas{K} \normL{\hat{\tensorOne{v}}}^2, \label{eq:affIneq_v_hv}                           \\
       & \normL{\hat{\tensorOne{v}}}^2 \le c \meas{K}^{-1} \normL{\tensorOne{v}}^2, \label{eq:affIneq_hv_v}                      \\
       & \normLg{\gamma(\tensorOne{v})}^2 \le c \meas{F} \normLg{\gamma(\hat{\tensorOne{v}})}^2, \label{eq:gamAffIneq_v_hv}      \\
       & \normLg{\gamma(\hat{\tensorOne{v}})}^2 \le c \meas{F}^{-1} \normLg{\gamma(\tensorOne{v})}^2, \label{eq:gamAffIneq_hv_v}
    \end{align}\label{eq:affIneq}
  \end{subequations}
\end{linenomath}
with $\hat{\tensorOne{v}} \in H^1(\hat{K})$ and $\tensorOne{v} \in H^1(K)$.
Finally, by combining equations \eqref{eq:normH_0} and \eqref{eq:invIneq}, we can bound the $H_1$-norm:
\begin{linenomath}
  \begin{subequations}
    \begin{align}
       & \normH{\tensorOne{v}}^2 \le (1+c) \diam{\Omega}^{-2} \normL{\tensorOne{v}}^2
    \end{align}\label{eq:normH_1}
  \end{subequations}
\end{linenomath}
\section*{Declaration of generative AI and AI-assisted technologies in the writing process}
During the preparation of this work the authors used \textit{gpt-4o} to improve language and readability. After using this tool/service, the authors reviewed and edited the content as needed and take full responsibility for the content of the publication.

\bibliographystyle{elsarticle-num}






\end{document}